\documentclass[11pt,leqno]{article}

\usepackage{amsthm,amsfonts,amssymb,amsmath,color}
\usepackage{amsbsy}
\numberwithin{equation}{section}

\renewcommand\d{\partial}
\renewcommand\a{\alpha}
\renewcommand\b{\beta}

\newcommand\R{\mathbb R}\newcommand\N{\mathbb N}

\def\g{\gamma}
\def\de{\mathfrak{z}}

\def\O{\Omega}
\def\th{\theta}

\def\l{\lambda}

\def\epsilon{\varepsilon}
\def\e{\varepsilon}

\newcommand\br{\begin{rem}}
\newcommand\er{\end{rem}}
\newcommand\bp{\begin{pmatrix}}
\newcommand\ep{\end{pmatrix}}
\newcommand\be{\begin{equation}}
\newcommand\ee{\end{equation}}
\newcommand\ba{\begin{equation}\begin{aligned}}
\newcommand\ea{\end{aligned}\end{equation}}

\newcommand\nn{\nonumber}

\usepackage{layout}

\newcommand{\CalE}{\mathcal{E}}

\newcommand{\RR}{{\mathcal R}}

\newcommand{\TT}{{\mathbb T}}

\newcommand{\II}{{\mathbb I}}

\newcommand{\SSS}{{\mathbb S}}

\newcommand{\YY}{{\mathbb Y}}

\newcommand{\tr}{{\rm tr }}

\newcommand{\vv}{{\mathbf v}}
\newcommand{\ff}{{\mathbf f}}

\newcommand{\vvarphi}{{\boldsymbol \varphi}}

\newcommand{\Ov}[1]{\overline{#1}}

\newcommand{\DC}{C^\infty_c}

\newcommand{\vr}{\varrho}

\newcommand{\tvr}{\tilde \vr}
\newcommand{\tvu}{\tilde \vu}

\newcommand{\vu}{\vc{u}}

\newcommand{\vc}[1]{{\bf #1}}
\newcommand{\Div}{{\rm div}}
\newcommand{\Grad}{\nabla_x}

\newcommand{\dx}{{\rm d} {x}}
\newcommand{\dt}{{\rm d} t }

\newcommand{\intO}[1]{\int_{\O} #1 \, \dx}

\newcommand{\vw}{\vc{w}}

\newcommand{\LLL}{\mathcal{L}^{-1}}

\newcommand{\teta}{\tilde \eta}
\newcommand{\tTT}{\tilde \TT}

\newtheorem{definition}{Definition}[section]
\newtheorem{theorem}[definition]{Theorem}
\newtheorem{proposition}[definition]{Proposition}
\newtheorem{lemma}[definition]{Lemma}

\newtheorem{remark}[definition]{Remark}

\def\ocirc#1{\ifmmode\setbox0=\hbox{$#1$}\dimen0=\ht0
    \advance\dimen0 by1pt\rlap{\hbox to\wd0{\hss\raise\dimen0
   \hbox{\hskip.2em$\scriptscriptstyle\circ$}\hss}}#1\else
   {\accent"17 #1}\fi}

\begin{document}

\title{Relative entropy, weak-strong uniqueness and conditional regularity for a compressible Oldroyd--B model}

\author{Yong Lu \footnote{Chern Institute of Mathematics  $\&$ LPMC, Nankai University, Tianjin 300071, China. Email: {\tt lvyong@amss.ac.cn}.} \and  Zhifei Zhang\footnote{School of Mathematical Science, Peking University, Beijing 100871, China. Email: {\tt zfzhang@math.pku.edu.cn}.}  \thanks{Z. Zhang is partially supported by NSF of China under
Grant No.11425103.}}

\date{}

\maketitle

\begin{abstract}

We consider the compressible Oldroyd--B model derived in \cite{Barrett-Lu-Suli}, where the existence of global-in-time finite energy weak solutions was shown in two dimensional setting. In this paper, we first state a local well-posedness result for this compressible Oldroyd--B model. In two dimensional setting, we give a (refined) blow-up criterion involving only the upper bound of the fluid density.  We then show that, if the initial fluid density and polymer number density admit a positive lower bound, the weak solution coincides with the strong one as long as the latter exists. Moreover, if the fluid density of a weak solution issued from regular initial data admits a finite upper bound, this weak solution is indeed a strong one; this can be seen as a corollary of the refined blow-up criterion and the weak-strong uniqueness.

\end{abstract}

{\bf Keywords:} compressible Oldroyd--B model; relative entropy inequality; weak-strong uniqueness; conditional regularity.

%{\bf MSC codes:}

% \tableofcontents

%%%%%%%%%%%%%%%%%%%%%%%%%%%%%%%%%%%%%%%%%%%%%%%%%%%%%%%%%%%%%%%%%%%%%%%%%%%%%%%%%%%%%%%%%%

\section{Introduction}

The incompressible Oldroyd--B model is a macroscopic model involving only macroscopic quantities, such as the velocity, the pressure and the stress. It is known that from the incompressible Navier--Stokes--Fokker--Planck system which is a micro-macro model describing incompressible dilute polymeric fluids, one can derive, at least formally, the Oldroyd--B model, see \cite{Bris-Lelievre}.

Similar derivation can be done in the compressible setting. Indeed,  in \cite{Barrett-Lu-Suli}, a compressible Oldroyd--B model was derived as a macroscopic closure of the compressible Navier--Stokes--Fokker--Planck equations studied in a series of papers by Barrett and S\"uli \cite{Barrett-Suli, Barrett-Suli1, Barrett-Suli2, Barrett-Suli4, BS2016}. % Also in \cite{Barrett-Lu-Suli}, the existence of global-in-time finite-energy weak solutions was shown in two dimensional setting.

We now represent the compressible Oldroy--B model derived in \cite{Barrett-Lu-Suli}. Let $\Omega \subset \mathbb{R}^d$ be a bounded open domain with $C^{2,\beta}$ boundary (briefly, a $C^{2,\beta}$ domain), with $\beta \in (0,1)$, and $d \in \{2,3\}$.  We consider the following compressible Oldroyd--B model posed in the time-space cylinder $(0,T)\times \O$:
\begin{alignat}{2}
\label{01a}
\d_t \vr + \Div_x (\vr \vu) &= 0,
\\
\label{02a}
\d_t (\vr\vu)+ \Div_x (\vr \vu \otimes \vu) +\nabla_x p(\vr)  - ( \mu \Delta_x \vu + \nu\nabla_x \Div_x \vu) &=\Div_x \big(\TT - (kL\eta + \de\, \eta^2)\,\II\, \big)  +  \vr\, \ff,
\\
\label{03a}
\d_t \eta + \Div_x (\eta \vu) &= \e \Delta_x \eta,
\\
\label{04a}
\d_t \TT + {\rm Div}_x (\vu\,\TT) - \left(\nabla_x \vu \,\TT + \TT\, \nabla_x^{\rm T} \vu \right) &= \e \Delta_x \TT + \frac{k}{2\lambda}\eta  \,\II - \frac{1}{2\lambda} \TT,
\end{alignat}
where the pressure $p$ and the density $\vr$ of the solvent are supposed to be related by the typical power law relation:
\be\label{pressure}
 p(\vr)=a \vr^\gamma, \quad a>0, \ \gamma > 1.
\ee
The term $\mu \Delta_x \vu + \nu \nabla_x \Div_x \vu$ corresponds to $\Div_x \SSS(\nabla_x \vu)$ where $\SSS(\nabla_x \vu)$ is the {\em Newtonian stress tensor} defined by
\be\label{Newtonian-tensor}
\SSS(\nabla_x \vu) = \mu^S \left( \frac{\nabla_x \vu + \nabla^{\rm T}_x \vu}{2} - \frac{1}{d} (\Div_x \vu) \II \right) + \mu^B (\Div_x \vu) \II,
\ee
where $\mu^S>0$ and $\mu^B\geq 0$ are shear and bulk viscosity coefficients, respectively. Indeed, direct calculation gives
\be\label{Newtonian-tensor-Lame}
\Div_x\SSS(\nabla_x \vu) = \frac{\mu^S}{2} \Delta_x \vu + \left(  \mu^B + \frac{\mu^S}{2} - \frac{\mu^S}{d} \right)\nabla_x \Div_x \vu = \mu \Delta_x \vu + \nu \nabla_x \Div_x \vu\nn
\ee
with
\be\label{mu-nu}
\mu:=\frac{\mu^S}{2}>0,\quad \nu:=\mu^B + \frac{\mu^S}{2} - \frac{\mu^S}{d} \geq 0.\nn
\ee

The velocity gradient matrix is defined as
\be\label{def-nabla-u}
( \nabla_x \vu )_{1\leq i,j\leq d}= (\d_{x_j} \vu_i)_{1\leq i,j\leq d}.\nn
%\quad \mbox{for any $1\leq i,j\leq d$}.
\ee

The symmetric matrix function $\TT = (\TT_{i,j})$, $1\leq i,j \leq d$, defined on $(0,T)\times \O$, is the extra stress tensor and the notation ${\rm Div}_x(\vu\,\TT)$ is defined by
\be\label{def-Div-tau}
\left({\rm Div}_x(\vu\,\TT)\right)_{i,j} = \Div_{ x}(\vu\,\TT_{i, j}), \quad 1\leq i,j \leq d.\nn
\ee
The meaning of the various quantities and parameters appearing in \eqref{01a}--\eqref{04a} were introduced in the derivation of the model later in \cite{Barrett-Lu-Suli}.  In particular, the parameters $\e$, $k$, $\l$ are all positive numbers, whereas $\de \geq 0$ and $L\geq 0$ with $\de + L > 0$.

The polymer number density $\eta$ is a nonnegative scalar function defined as the integral of the probability density function $\psi$, which is governed by the Fokker--Planck equation, in the conformation vector which is a microscopic variable in the modelling of dilute polymer chains. The term $q(\eta):= kL \eta + \de \eta^2$ in the momentum equation \eqref{02a} can be seen as the \emph{polymer pressure}, compared to the fluid pressure $p(\vr)$.

The equations \eqref{01a}--\eqref{04a} are supplemented by initial conditions for $\vr$, $\vu$, $\eta$ and $\TT$, and the following boundary conditions are imposed:
\begin{alignat}{2}
\label{05a}
\vu&=\mathbf{0}   &&\quad \mbox{on}\ (0,T)\times \partial\Omega,
\\
\label{06a}
\d_{\bf n} \eta  &=0 &&\quad \mbox{on}\  (0,T)\times \d\O,
\\
\label{07a}
\d_{\bf n} \TT &=0 &&\quad \mbox{on}\  (0,T)\times \d\O.
\end{alignat}
Here $\d_{\bf n}:= {\bf n}\cdot \nabla_x$, where ${\bf n}$ is the outer unit normal vector on the boundary $\d\O$.
The external force $\ff$ is assumed to be in $L^\infty((0,T)\times \O;\R^d)$.

There are stress diffusion terms $\e \Delta_x \eta$ and $\e \Delta_x \TT$ in our model. Such spatial stress diffusions are indeed allowed in some modeling of complex fluids, such as in the creeping flow regime as pointed out in \cite{CK}. Also in the modeling of the compressible Navier--Stokes--Fokker--Planck system arising in the kinetic theory of dilute polymeric fluids, where polymer chains immersed in a barotropic, compressible, isothermal, viscous Newtonian solvent, Barrett and S\"uli \cite{Barrett-Suli} observed the presence of the centre-of-mass diffusion term $\e \Delta_x \psi$, where $\psi$ is the probability density function depending on both microscopic and macroscopic variables; as a result, its macroscopic closure contains such diffusion terms.

The incompressible Oldroyd--B model  attracts continuous attentions of mathematicians. The local-in-time well-posedness, as well as the global-in-time well-posedness with small data, in various spaces is known due to the contribution of Renardy \cite{Renardy90}, Guillop\'e and Saut \cite{G-S90-1, G-S90-2} and  Fern\'andez-Cara, Guill\'en and Ortega \cite{F-G-O02}. Concerning the global-in-time existence of solutions with large data, under the corotational derivative setting where for the system in the extra stress tensor, a term related to the velocity gradient is replaced by its anti-symmetric part, Lions and Masmoudi \cite{LM00} showed the global-in-time existence of weak solutions. In the presence of the stress diffusion for which there arises a regularization Laplacian term in the extra stess tensor, Barrett and Boyaval \cite{Barrett-Boyaval} showed the global-in-time existence of weak solutions in two dimensional setting. Also in the presence of the stress diffusion and in two dimensional setting, Constantin and Kliegl \cite{CK} proved the global existence of strong solutions, which can be seen as a continuation of the global well-posedness theory for two dimensional incompressible Navier--Stokes equations.

While, many fundamental problems for the incompressible Oldroyd--B model are still open, such as the global-in-time existence of large data solutions, even the weak ones, both in two dimensional and three dimensional setting without the stress diffusion. Even with the stress diffusion, the global-in-time existence of large data solutions, strong or weak, is still open  in three dimensional setting. This is somehow within the expectation, since the well-posedness of the incompressible Navier--Stokes equations is a well-known open problem.

\smallskip

Even less are known concerning compressible Oldroyd--B models. Let us mention some mathematical  results for compressible viscoelastic models, which  has been the subject of active research in recent years. The existence and uniqueness of local strong solutions and the existence of global solutions near equilibrium for macroscopic models of three-dimensional compressible viscoelastic fluids was considered in \cite{Qian-Zhang, Hu-Wang3,Hu-Wu, Lei}. In particular, Fang and Zi \cite{Fang-Zi} proved the existence of a unique local-in-time strong solution to a compressible Oldroyd--B model and established a blow-up criterion for strong solutions.  In \cite{Barrett-Lu-Suli}, the existence of global-in-time weak solutions in two dimensional setting for the compressible Oldroyd--B model \eqref{01a}--\eqref{07a} was shown.

%We remark that the compressible Oldroyd--B models considered before, like in \cite{Fang-Zi, Lei}, are basically a modification of the classical incompressible Oldroyd--B model by replacing the incompressible Navier--Stokes equations by the compressible ones. As a result, the equation in the polymer number density $\eta$ is not included.  The model considered in this paper is derived from the micro-macro models of dilute polymeric fluids under Hookean dumbbell model setting and thus has a strong physical background. % Moreover, this fully macroscopic model derived from the micro-macro models coincides with the formal derivation of the classical incompressible Oldroyd--B model from the incompressible Navier--Stokes--Fokker--Planck equations.
%In our model \eqref{01a}--\eqref{04a}, there are diffusion terms $\e \Delta_x \TT$ and $\e \Delta_x \eta$, while in earlier studies of compressible viscoelastic fluid flow models, like \cite{Fang-Zi, Lei}, such diffusion terms were not included. %, which usually makes the mathematical study more difficult.
%As in the derivation recalled earlier and the study of Constantin and Kliegl \cite{CK}, the presence of such diffusion terms does be allowed in some models.

\section{Main results}

In this section, we state our main results. We first recall the result shown in \cite{Barrett-Lu-Suli} concerning the global-in-time existence of weak solutions. We state the theorem concerning the local well--posedness of strong solutions and a blow-up criterion.  %this theorem is given by summarizing Theorem 1.1 and Theorem 1.2 in \cite{Fang-Zi}, and can be proved similarly as in \cite{Fang-Zi}, so we omit the proof.
In two dimensional setting, we give a refined blow-up criterion result where only the $L^\infty$ bound of the fluid density is needed. We then show a weak-strong uniqueness result by using the relative entropy method. As a corollary, this offers us a conditional regularity theorem.

\subsection{Global-in-time finite energy weak solutions}

We give our basic hypotheses on the initial data:
\ba\label{ini-data-f}
&\vr(0,\cdot) = \vr_0(\cdot) \ \mbox{with}\ \vr_0 \geq 0 \ {\rm a.e.} \ \mbox{in} \ \O, \quad \vr_0 \in L^\gamma(\O),\\
&\vu(0,\cdot) = \vu_0(\cdot) \in L^r(\O;\R^d) \ \mbox{for some $r\geq 2\g'$}\ \mbox{such that}\ \vr_0|\vu_0|^2 \in L^1(\O),\\
&\eta(0,\cdot)=\eta_0(\cdot) \ \mbox{with}\ \eta_0 \geq 0 \ {\rm a.e.} \ \mbox{in} \ \O, \quad \left\{\begin{aligned}
&\eta_0 \in L^2 (\O), \quad \mbox{if}\ \de >0,\\
&\eta_0 \log \eta_0 \in L^1 (\O), \quad \mbox{if} \ \de=0,
\end{aligned}\right. \\ %\int_\Omega \eta_0(x)\,\dx = 1,\\
&\TT(0,\cdot) = \TT_0(\cdot) \ \mbox{with}\ \TT_0=\TT_0^{\rm T} \geq 0 \ \mbox{a.e. in}  \ \O ,\quad  \TT_0 \in L^2(\O;\R^{d \times d}).
\ea
A related weak solution is defined as follows:
\begin{definition}\label{def-weaksl-f} Let $T>0$ and suppose that $\O\subset \R^d$ is a bounded $C^{2,\beta}$ domain with  $0<\beta<1$. We say that $(\vr,\vu,\eta,\TT)$ is a finite-energy
weak solution in $(0,T)\times \O$ to the system of equations \eqref{01a}--\eqref{07a},  supplemented by the initial data \eqref{ini-data-f}, if:
\begin{itemize}
\item $\vr \geq 0 \ {\rm a.e.\  in} \ (0,T) \times \Omega$, $\vr \in  C_w ([0,T];  L^\gamma(\Omega))$,  $\vu\in L^{2}(0,T;W_0^{1,2}(\Omega; \R^d))$,
\vspace{-2mm}
\begin{align*}
&\vr \vu \in C_w([0,T]; L^{\frac{2 \gamma}{ \gamma + 1}}(\Omega; \R^d)),\quad \vr |\vu|^2 \in L^\infty(0,T; L^{1}(\Omega)),\\
&\eta  \geq 0  \ {\rm a.e.\  in} \ (0,T) \times \Omega,\\
& \left\{\begin{aligned}
&\eta \in C_w ([0,T];  L^2(\Omega)) \cap L^2 (0,T;  W^{1,2}(\Omega)), \quad \mbox{if}\ \de >0,\\
&\eta \log \eta \in L^\infty(0,T;L^1(\O)), \ \eta^{\frac{1}{2}} \in L^2(0,T;W^{1,2}(\O)), \ \eta \in C_w([0,T];L^1(\O)), \quad \mbox{if} \ \de=0,
\end{aligned}\right. \\
&{ \TT = \TT^{\rm T} \geq 0  \ {\rm a.e.\  in} \ (0,T) \times \Omega,\quad \TT \in C_w ([0,T];  L^2(\Omega;\R^{d \times d})) \cap L^2 (0,T;  W^{1,2}(\Omega;\R^{d \times d}))}.
\end{align*}
\item For any $t \in (0,T)$ and any test function $\phi \in C^\infty([0,T] \times \Ov{\Omega})$, one has
\be\label{weak-form1-f}
\int_0^t\intO{\big[ \vr \partial_t \phi + \vr \vu \cdot \Grad \phi \big]} \,\dt' =
\intO{\vr(t, \cdot) \phi (t, \cdot) } - \intO{ \vr_{0} \phi (0, \cdot) },
\ee
\be\label{weak-form2-f}
\int_0^t \intO{ \big[ \eta \partial_t \phi + \eta \vu \cdot \Grad \phi - \e \nabla_x\eta \cdot \nabla_x \phi \big]} \, \dt' =  \intO{ \eta(t, \cdot) \phi (t, \cdot) } - \intO{ \eta_{0} \phi (0, \cdot) }.
\ee
\item For any $t \in (0,T)$ and any test function $\vvarphi \in C^\infty([0,T]; \DC({\Omega};\R^d))$, one has
\ba\label{weak-form3-f}
&\hspace{-0.7cm}\int_0^t \intO{ \big[ \vr \vu \cdot \partial_t \vvarphi + (\vr \vu \otimes \vu) : \Grad \vvarphi  + p(\vr)\, \Div_x \vvarphi + \big(kL\eta+\de\,\eta^2\big)\, \Div_x \vvarphi  - \SSS(\nabla_x \vu) : \Grad \vvarphi \big] } \, \dt'\\
&\hspace{-0.7cm}= \int_0^t \intO{ \big[\TT : \nabla_x\vvarphi - \vr\, \ff \cdot \vvarphi \big]} \, \dt' + \intO{ \vr \vu (t, \cdot) \cdot \vvarphi (t, \cdot) } - \intO{ \vr_{0} \vu_{0} \cdot \vvarphi(0, \cdot) }.
\ea
\item For any $t \in (0,T)$ and any test function $\YY \in C^\infty([0,T] \times \Ov{\Omega};\R^{d \times d})$, one has
\ba\label{weak-form4-f}
&\int_0^t \intO{ \left[ \TT : \partial_t \YY +  (\vu \,\TT) :: \nabla_{x} \YY
+ \left(\nabla_x \vu \,\TT + \TT\, \nabla_x^{\rm T} \vu \right):\YY
- \e \nabla_x  \TT :: \nabla_x \YY  \right]}\,\dt'\\
&=\int_0^t \intO{ \left[ - \frac{k}{2\lambda}\eta  \,\tr\left(\YY\right) + \frac{1}{2\lambda} \TT:\YY \right]} \, \dt' + \intO{ \TT (t, \cdot) : \YY (t, \cdot) }
- \intO{ \TT_{0} : \YY(0, \cdot) }.
\ea
\item The continuity equation holds in the sense of renormalized solutions: for any $b\in C^1_b[0,\infty),$
\be\label{weak-renormal-f}
\d_t b(\vr) +\Div_x (b(\vr)\vu) + (b'(\vr)\vr - b(\vr))\,\Div_x \vu =0 \quad \mbox{in} \ \mathcal{D}' ((0,T)\times \O).
\ee
%

%
%\be\label{cond-b-renormal-f}
%|b'(s)|<C \,s^{-\lambda_0}\quad \forall \,s\in (0,1] \qquad
%\mbox{and} \qquad |b'(s)|<C \,s^{\lambda_1}\quad \forall \,s \geq 1,
%\ee
%where $\lambda_0 <1$ and $\lambda_1 \in (-1,\infty)$; see (6.2.9) and (6.2.10) in
%\cite{N-book}.
%
\item  For a.e. $t \in (0,T)$, the following \emph{energy inequality} holds:
\ba\label{energy1-f}
&\int_\O \left[ \frac{1}{2} \vr |\vu|^2 + \frac{a}{\g-1} \vr^\g + \left(k L  (\eta \log \eta + 1) + \de \,\eta^2\right)+ \frac{1}{2}\tr \left(\TT \right)\right](t,\cdot)\dx \\
&\quad+ 2\e\int_0^t \int_\O 2k L   |\nabla_x \eta^{\frac12}|^2 +  \de\,  |\nabla_x \eta|^2 \,\dx\,\dt'  + \frac{1}{4\l} \int_0^t  \int_\O \tr\left(\TT\right) \dx \,\dt'  \\
&\quad +  \int_0^t  \int_\O \mu \left| \nabla_x \vu \right|^2 + \nu |\Div_x \vu|^2\,\dx \,\dt'  \\
& \leq  \int_\O \left[ \frac{1}{2} \vr_0 |\vu_0|^2 + \frac{a}{\g-1} \vr_0^\g + \left(k L  (\eta_0 \log \eta_0 + 1) + \de \,\eta_0^2\right)+ \frac{1}{2}\tr \left(\TT_0\right) \right]\dx \\
&\quad + \int_0^t \int_\O \vr\,\ff \cdot  \vu \,\dx\,\dt' + \frac{k \,d}{4\lambda} \int_0^t \int_\O \eta  \,\dx \,\dt'.
\ea
\end{itemize}
\end{definition}

\smallskip

We recall the associated result concerning the existence of large data global-in-time finite-energy weak solutions, which can be obtained by summarizing Theorem 11.2 and Theorem 12.1 in \cite{Barrett-Lu-Suli}.

\smallskip

\begin{theorem}\label{thm-weak}
Let $\g>1$ and $\Omega \subset \mathbb{R}^2$ be a bounded $C^{2,\beta}$ domain with $\beta \in (0,1)$. Assume the parameters $\e$, $k$, $\l$ are all positive numbers and  $\de \geq 0$, $L\geq 0$ with $\de + L >0$. Then for any $T>0$, there exists a finite-energy weak solution $(\vr,\vu,\eta,\TT)$ in the sense of Definition \ref{def-weaksl-f} with initial data \eqref{ini-data-f}. Moreover, the extra stress tensor $\TT $ satisfies the bound
\ba\label{est-tau-higher}
 \int_\O  | \TT(t, \cdot)|^2 \,\dx + \e  \int_0^t \int_\O |\nabla_x \TT|^2  \,\dx\,\dt' + \frac{1}{4\l}\int_0^t \int_\O | \TT|^2 \dx\, \dt'  \leq C(t, E_0),
\ea
for a.e. $t \in (0,T)$, where $E_0$ is given by
$$
E_0:=\int_\O \left[ \frac{1}{2} \vr_0 |\vu_0|^2 + \frac{a}{\g-1} \vr_0^\g + \left(k L  (\eta_0 \log \eta_0 + 1) + \de \,\eta_0^2\right)+ \frac{1}{2}\tr \left(\TT_0\right) \right]\dx.
$$
\end{theorem}

%Note that this is a result in two dimensional setting. For domains in $\R^3$, the global-in-time existence of solutions is still an open problem, even for weak solutions.

\subsection{Local well-posedness and blow-up criterion}

We now state the result concerning the local-in-time existence of strong solutions. By strong solution, here we mean a weak solution satisfy the equations \eqref{01a}--\eqref{07a} $a.e$ in the space-time cylinder under consideration.

\begin{theorem}\label{thm-strong} Let $\g>1$ and $\O\subset \R^d$ be a bounded $C^{2,\beta}$ domain with $\beta \in (0,1)$. Assume the parameters $\e$, $k$, $\l$ are all positive numbers, whereas $\de \geq 0$ and $L\geq 0$ with $\de + L > 0$. We assume the external force $\ff\in W^{1,2}((0,\infty)\times \O)$. Additional to the assumption on initial data in \eqref{ini-data-f}, we suppose
\be\label{ini-data-s}
\vr_0\in W^{1,6}(\O),\  \eta_0 \in W^{2,2}_{\bf n}(\O), \  \TT_0 \in W_{\bf n}^{2,2}(\O;\R^{d\times d}), \ \vu_0 \in W_0^{1,2}\cap W^{2,2}(\O;\R^d),
\ee
where the notation $W^{2,2}_{\bf n}(\O):=\{f\in W^{2,2}(\O):\ \d_{\bf n} f =0 \ \mbox{on} \ \d\O\}$. Suppose there holds
  \be\label{constrain-data}
  - ( \mu \Delta_x \vu_0 + \nu\nabla_x \Div_x \vu_0) + \nabla_x p(\vr_0) - \Div_x \TT_0 +\nabla_x (kL\eta_0 + \de\, \eta_0^2) = \sqrt{\vr_0}  g
  \ee
for some $g\in L^2(\O;\R^d)$. Thus, there exists a unique strong solution $(\vr, \vu, \eta, \TT)$ to \eqref{01a}--\eqref{07a} with a maximal existence time $T_*\in (0,\infty]$ such that
\ba\label{est-strong}
&\vr\geq 0, \ \vr \in C([0,T_*), W^{1,6}(\O)),\\
& \vu \in C([0,T_*), W_0^{1,2}\cap W^{2,2}(\O;\R^d)) \cap L^2_{\rm loc}([0,T_*); W^{2,r}(\O;\R^d)),\\
&\eta \geq 0, \ \TT=\TT^{\rm T} \geq 0, \ (\eta, \TT)  \in C([0,T_*), W^{2,2}_{\bf n})\cap L^2_{\rm loc}([0,T_*);W^{3,2})(\O;\R\times \R^{d\times d})) ,
\ea
where $r=6$ when $d=3$ and $r\in (1,\infty)$ is arbitrary when $d=2$.

If $T_* < \infty$, the following quantity blow-up:
\be\label{criterion1}
\limsup_{T\to T_*} \big( \| \vr \|_{L^\infty((0,T)\times \O )}+ \| \eta \|_{L^\infty((0,T)\times \O )} +  \| \TT \|_{L^2(0,T;L^\infty(\O;\R^{d\times d}))}  \big) =\infty.
\ee

\end{theorem}

\begin{remark}\label{rem-thm-strong1}
This local-in-time well-posedness result for strong solution is inspired by the study \cite{CK03} for compressible Navier--Stokes equations.  If the initial density is additionally assumed to have a positive lower bound, condition \eqref{constrain-data} is automatically satisfied. In such a setting, local-in-time well-posedness can be also obtained by employing the method in \cite{Solonnikov80,Valli83}.

\end{remark}

\begin{remark}\label{rem-thm-strong2}

Theorem \ref{thm-strong} is given in a similar manner as Theorem 1.1 and Theorem 1.2 in \cite{Fang-Zi} and can be proved similarly. In fact, there are extra diffusion terms in our model compared to the model considered in \cite{Fang-Zi}, and this makes the proof even easier. Hence, we omit the proof of this theorem. The regularity assumption on initial data may be not optimal and can be relaxed accordingly by employing the argument in \cite{CK03}.

\end{remark}

In two dimensional setting, we offer the following refined blow-up criterion:

\begin{theorem}\label{thm-criterion} Let $d=2$ and $(\vr, \vu, \eta, \TT)$ be the strong solution obtained in Theorem \ref{thm-strong} to \eqref{01a}--\eqref{07a} with a maximal existence time $T_*\in (0,\infty]$. If $T_* < \infty$, there holds
\be\label{criterion2}
\limsup_{T\to T_*} \left( \| \vr \|_{L^\infty((0,T)\times \O )} \right) =\infty.
\ee
\end{theorem}

\begin{remark}\label{rem-criterion1}

The blow-up criterion \eqref{criterion1}, which is reproduced from \cite{Fang-Zi}, is inspired by the related study for compressible Navier--Stokes equations in \cite{SZ11,SWZ11} and for incompressible Oldroyd--B model in \cite{CN01}. Our refined one in Theorem \ref{thm-criterion} coincides with those in \cite{SZ11,SWZ11} where only the upper bound of the fluid density is needed. Such a refinement crucially depends on the two dimensional setting and the presence of the diffusions in $\TT$ and $\eta$. This setting allows to obtain improved estimates in $\TT$ and $\eta$ that are uniform in time (see Proposition \ref{lem-eta-better} and Proposition \ref{lem-TT-better} later on). %Then by careful analysis, we may get rid of the restriction on $\TT$ and $\eta$ in \eqref{criterion1}.
In three dimensional setting, it is not known whether one can get such an improvement.

\end{remark}

\subsection{Weak-strong uniqueness and conditional regularity}

Still in two dimensional setting, we show the following weak-strong uniqueness result, provided the fluid density and polymer number density admit a positive lower bound.

\begin{theorem}\label{thm-ws} Let $d=2$. Let $(\vr, \vu, \eta, \TT)$ be a finite energy weak solution obtained in Theorem \ref{thm-weak} and let $(\tvr, \tvu, \teta, \tTT)$ be the strong solution obtained in Theorem \ref{thm-strong} with the same initial data satisfying the assumption in Theorem \ref{thm-strong}. If in addition the initial data satisfy
\be\label{bound-lower}
 \inf_{ \O} \vr_0 >0,\quad \inf_{\O} \eta_0 > 0,
\ee
then there holds
\be\label{weak=strong}
(\vr, \vu, \eta, \TT)=(\tvr, \tvu, \teta, \tTT)\quad  \mbox{in} \quad [0,T_*)\times \O.
\ee

\end{theorem}

Finally, as a corollary of Theorem \ref{thm-criterion} and Theorem \ref{thm-ws}, we have the following conditional regularity result for finite energy weak solutions.
\begin{theorem}\label{thm-con-reg} Let $d=2$. Let $(\vr, \vu, \eta, \TT)$ be a finite energy weak solution obtained in Theorem \ref{thm-weak} with initial data satisfying the assumption in Theorem \ref{thm-strong} and the additional lower bound constrain \eqref{bound-lower}. If for some $T>0$ there holds the upper bound
\be\label{bound-con-reg}
\sup_{(0,T)\times \O} \vr <\infty,
\ee
then the weak solution $(\vr, \vu, \eta, \TT)$ is actually a strong one satisfying the estimates \eqref{est-strong} over time interval $[0,T]$.

\end{theorem}

\medskip

The rest of the paper is devoted to the proof of Theorems \ref{thm-criterion}, \ref{thm-ws} and  \ref{thm-con-reg}. In Section \ref{sec-prelim}, we recall some necessary lemmas.  Theorems \ref{thm-criterion}, \ref{thm-ws} and  \ref{thm-con-reg} are proved in Sections \ref{sec-thm-criterion} and \ref{sec-thm-ws-unique}, respectively. Throughout the paper, $C$ denotes some uniform constant of which the value may differ from line to line.

\section{Preliminaries}\label{sec-prelim}

In this section we recall some technical tools that will be required in the paper. The first one considers the Dirichlet problems of Lam\'e systems:
\begin{lemma}\label{lem-Lame}
Let $\mu>0, \ \nu\geq 0,$ $G \subset \R^d$ be a bounded $C^{2,\b}$ domain with $\b\in (0,1)$. Let $\vu$ be the unique weak solution to
\be\label{sys-Lame}
\left\{\begin{aligned} & \mathcal{L}\vu:=- \mu \Delta_x \vu - \nu \nabla_x \Div_x\vu = \ff, \quad &&\mbox{in}\ G,\\
&\vu =0, \quad &&\mbox{on}\ \d G,
\end{aligned}\right.
\ee
for some $\ff$ as following.
\begin{itemize}
\item[(i).] If $\ff\in L^r(G), \ 1<r<\infty$, then $$\|\vu\|_{W_0^{1,r}\cap W^{2,r}(G)} \leq C(\mu, r,d,G) \|\ff\|_{L^{r}(G)}. $$

\item [(ii).] If $\ff = \Div_x F$ with $F = (F_{i,j})_{1\leq i,j\leq d} \in L^r(G), \ 1<r<\infty$, then $$\|\vu\|_{W_0^{1,r}(G)} \leq C(\mu, r,d,G) \|F\|_{L^{r}(G)}. $$

    \item [(iii).] If $\ff = \Div_x F$, $F = (F_{i,j})_{1\leq i,j\leq d}$ with $F_{i,j} = \Div_x H_{i,j}$ with $H = (H_{i,j,k})_{1\leq i,j,k\leq d} \in L^r(G), \ 1<r<\infty$, then $$\|\vu\|_{L^{r}(G)} \leq C(\mu, r,d,G) \|H\|_{L^{r}(G)}. $$

%\item [(ii).] If $\ff = \Div_x F$ with $F = (F_{i,j})_{1\leq i,j\leq d} \in L^\infty(G)$, then $\nabla_x \vu \in {\rm BMO}(G)$ and $$\|\nabla_x\vu\|_{{\rm BMO}(G)} \leq C(\mu,d,G) \|\ff\|_{L^{\infty}(G)}. $$

\end{itemize}

\end{lemma}
In the sequel, we will use $\mathcal{L}^{-1}(\ff)$ to denote the solution $\vu$ to \eqref{sys-Lame}. %It is clear that $\mathcal{L}^{-1}$ is a linear operator in some proper space.

\medskip

We then recall the following Gagliardo--Nirenberg inequality
 \begin{lemma}\label{lem-GN}
 Let $G\subset\R^d$ be a bounded Lipschitz domain; then, for any $r\in [2,\infty)$ if $d=2$, and $r\in [2,6]$ if $d=3$, one has, for any $v\in W^{1,2}(G)$, that:
\be\label{G-N-ineq}
\|v\|_{L^r(G)} \leq C(r,d,G) \|v\|_{L^2(G)}^{1-\th} \|v\|_{W^{1,2}(G)}^\th,\quad \th:=d\big(\frac{1}{2}-\frac{1}{r}\big).\nn
\ee

\end{lemma}

%We then recall the following Brezis--Waogner type inequality:

%\begin{lemma}\label{lem-BW}
%Let $G \in \R^d$ be a bounded Lipschitz domain. Let $v\in W^{1,r}(G)$ with $r>d$. Then
%$$
%\|v\|_{L^\infty(G)} \leq C\left( 1+ \|v\|_{{\rm BMO}(G)} \right) \log \left(e + \|\nabla v\|_{L^r(G)}\right).
%$$
%\end{lemma}

Lemma \ref{lem-Lame} is a collection of classical estimates for linear elliptic systems. The Gagliardo--Nirenberg inequality is rather classical. We refer to the books \cite{Galdi-book, N-book} for the proof.

%\subsection{Regularity of the parabolic Neumann problem}

\medskip

Now we recall some regularity results for parabolic Neumann problems. We first introduce fractional-order Sobolev spaces. Let $G$ be the whole space $\R^d$ or a bounded Lipschitz domain in $\R^d$. For any $k\in \N$, $\b \in (0,1)$ and $s\in [1,\infty) $, we define
\[
%\be\label{def-frac-Sob}
 W^{k+\b,s}(G):=\big\{v\in W^{k,s}(G) : \| v \|_{W^{k+\b,s}(G)}<\infty \big\},
%\ee
\]
where
$$
\| v \|_{W^{k+\b,s}(G)}:=\| v \|_{W^{k,s}(G)} + \sum_{|\a|=k} \left(\int_G\int_G
\frac{|\d^\a v(x)- \d^\a v(y)|^s}{|x-y|^{d+\b s}} \, \dx \,{\rm d}y\right)^{\frac{1}{s}}.
$$
%We recall the following classical compact embedding theorem (see Theorem 7.1 in \cite{NPV}).
%\begin{lemma}\label{lem-frac-sob} Let $\O\subset \R^d$ be a bounded Lipschitz domain and suppose that $k\in \N$, \ $\b \in (0,1)$ and $s\in [1,\infty)$; then, the embedding
%of $W^{k+\b,s}(\O)$ into $W^{k,s}(\O)$ is compact.
%\end{lemma}

The following classical results are taken from Section 7.6.1 in \cite{N-book}. Consider the parabolic initial-boundary value problem:
\ba\label{para-ib-Neumann}
&\d_t \rho - \e\, \Delta_x \rho = h \ \mbox{in} \ (0,T)\times G;
\quad \rho(0,\cdot) = \rho_0 \ \mbox{in} \   G; \quad \d_{\bf n}\rho =0 \ \mbox{in} \ (0,T)\times \d G.
\ea
Here $\e>0$, $\rho_0$ and $h$ are known functions, and $\rho$ is the unknown solution. The first regularity result of relevance to us here is the following lemma.

\smallskip

\begin{lemma}\label{lem-parabolic-1}
Let $0<\beta < 1, \ 1<p,q<\infty$,  $G\subset \R^d$ be a bounded $C^{2,\beta}$ domain with $\beta\in (0,1)$,
$$
\rho_0 \in W^{2-\frac{2}{p},q}_{\bf n},\quad h\in L^p(0,T;L^q(G)),
$$
where $W^{2-\frac{2}{p},q}_{\bf n}$ is the completion of the linear space $\{v \in  C^\infty(\overline G): \ \d_{\bf n} v|_{\d G} = 0 \}$ with respect to norm $W^{2-\frac{2}{p},q}(G)$.
Then, there exists a unique function $\rho$ satisfying
$$
\rho \in L^p(0,T;W^{2,q}(G)) \cap C([0,T];W^{2-\frac{2}{p},q}(G)), \quad \d_t \rho \in L^p(0,T;L^q(G))
$$
solving \eqref{para-ib-Neumann} in $(0,T)\times G$; in addition, $\rho$ satisfies the Neumann boundary condition in $\eqref{para-ib-Neumann}$ in the sense of the normal trace, which is well defined since $\Delta_x \rho \in L^p(0,T;L^{q}(G))$.
Moreover,
\begin{align*}
%&
\e^{1-\frac{1}{p}}\|\rho\|_{L^\infty(0,T;W^{2-\frac{2}{p},q}(G))}
+ \left\|\partial_t\rho \right\|_{L^p(0,T; L^q(G))}
+ \e \|\rho\|_{L^p(0,T; W^{2,q}(G))}  \nonumber\\
%&\qquad\qquad
\leq C(p,q,G)\big[\e^{1-\frac{1}{p}}
\|\rho_0\|_{W^{2-\frac{2}{p},q}(G)} + \|h\|_{L^p(0,T;L^q(G))}\big].
\end{align*}

\end{lemma}

\smallskip

The second result concerns parabolic problems with a divergence-form source term
$h=\Div_x \bf g$.

\smallskip

\begin{lemma}\label{lem-parabolic-2}
Let $0<\beta < 1, \ 1<p,q<\infty$, $G\subset \R^d$ be a bounded $C^{2,\beta}$ domain with $\beta\in (0,1)$,
$$
\rho_0 \in L^{q}(G),\quad {\bf g} \in L^p(0,T;L^q(G;\R^d)).
$$
Then, there exists a unique solution $\rho\in L^p(0,T;W^{1,q}(G)) \cap C([0,T];L^{q}(G)) $ satisfying
$$
\vr(0,\cdot) = \vr_0(\cdot) \ a.e. \  \mbox{in $G$}, \quad \frac{\rm d}{\dt} \int_G \rho\, \phi\,\dx + \e \int_G \nabla_x\rho \cdot \nabla_x \phi \,\dx = - \int_G  {\bf g} \cdot \nabla_x \phi \,\dx \quad \mbox{in}\ {\mathcal D}'(0,T),
$$
for any $\phi \in C^\infty(\overline G)$. Moreover,
\begin{align*}
\e^{1-\frac{1}{p}}\|\rho\|_{L^\infty(0,T;L^q(G))}
+ \e \|\nabla_x \rho\|_{L^p(0,T;L^q(G;\R^d))} \leq C(p,q,G)
\big[\e^{1-\frac{1}{p}}\|\rho_0\|_{L^q(G)}
+ \|{\bf g}\|_{L^p(0,T;L^q(G;\R^d))}\big].
\end{align*}
\end{lemma}

%\subsection{Bogovski{\u\i} operator}

\medskip

Finally, we recall the Bogovski{\u\i} operator, whose construction can be found in \cite{bog} and in Chapter III of Galdi's book \cite{Galdi-book}.
%\smallskip
\begin{lemma}\label{lem-bog}
 Let $1<p<\infty$ and  $G\subset \R^d$ be a bounded  Lipschitz domain. Let $L^{p}_0(G)$ be the space of all $L^p(G)$ functions with zero mean value. Then, there exists a linear operator ${\mathcal B}_G$ from $L_0^p(G)$ to $W_0^{1,p}(G;\R^d)$ such that for any $\rho \in L_0^p(G)$ one has
$$
\Div_x \mathcal{B}_G (\rho) =\rho \quad \mbox{in} \ G; \quad \|\mathcal{B}_G (\rho)\|_{W_0^{1,p}(G;\R^d)} \leq c (p,d,G)\, \|\rho\|_{L^p(G)}.
$$
%where the constant $c$ depends only on $p,\ d$ and the Lipschitz character of $G$.

If, in addition, $\rho=\Div_x {\bf g}$ for some ${\bf g} \in L^{q}(G;\R^d), \  1<q<\infty, \  {\bf g} \cdot {\bf n}  = 0 $ on $\partial G$, then the following inequality holds:
$$
\|\mathcal{B}_G (\rho)\|_{L^{q}(G;\R^d)} \leq c (d, q,G)\, \| {\bf g} \|_{L^q(G;\R^d)}.
$$
\end{lemma}

\section{A refined blow-up criterion}\label{sec-thm-criterion}

This section is devoted to the proof of Theorem \ref{thm-criterion}. Let $d=2$ and $(\vr,\vu, \eta, \TT)$ be the strong solution given in Theorem \ref{thm-strong} with $T_*$ the maximal existence time. By contradiction argument, to prove Theorem \ref{thm-criterion}, it is sufficient to show that, if
\be\label{ass-criterion-bd}
T_*<\infty, \quad \limsup_{T\to T_*} \left( \| \vr \|_{L^\infty((0,T)\times \O )} \right) <\infty,
\ee
there holds the following uniform estimates in $\eta$ and $\TT$:
\be\label{criterion-bd}
\limsup_{T\to T_*} \left(  \| \eta \|_{L^\infty((0,T)\times \O )} +  \| \TT \|_{L^2((0,T),L^\infty(\O;\R^{2\times 2}))}   \right) <\infty.
\ee

\medskip

Let $T_1\in (0,T_*)$  be close to $T_*$ and be determined later on. By Theorem \ref{thm-strong}, there holds
\ba\label{est-strong-T1}
\|\vr\|_{L^\infty(0,T_1; W^{1,6}(\O))}   + \|(\vu, \eta, \TT)\|_{L^\infty(0,T_1; W^{2,2}(\O;\R^2 \times \R\times \R^{2\times 2}))} \leq C(T_1,T_*)<\infty.
\ea
Remark that the constant $C(T_1,T_*)$ may be unbounded as $T_1 \to T_*.$ %However, $T_1$ shall be fixed later on.

By Sobolev embedding inequality $\|\vr\|_{L^\infty(\O)} \leq C \|\vr\|_{W^{1,6}(\O)}$ and estimates \eqref{est-strong-T1}, the assumption \eqref{ass-criterion-bd} is equivalent to the following assumption
\be\label{ass-criterion-bd2}
\| \vr \|_{L^\infty((0,T_*)\times \O )} <\infty.
\ee
Now, starting from \eqref{ass-criterion-bd2}, we show our desired estimate \eqref{criterion-bd} step by step in the rest of this section. Some ideas are based on the methods in \cite{Hoff95, SZ11}, while quite a few technical difficulties coming from the terms in $\eta, \ \TT$ need to be handled.

In the sequel, to avoid notation complicity, we sometimes simply use $L^r(0,T;X(\O))$ to denote scalar function spaces $L^r(0,T;X(\O))$, vector valued function spaces $L^r(0,T;X(\O;\R^n))$ or matrix valued function spaces $L^r(0,T;X(\O;\R^{n\times n}))$ if there is no confusion. In the rest of this section, the constant $C$ depends on only the initial data and  the quantity $\| \vr \|_{L^\infty((0,T_*)\times \O )}$ which is assumed to be bounded by contradiction.

\subsection{A priori estimates}\label{sec-a-priori}

In this section and next section (Section  \ref{sec-higher-eta}), we give some estimates that are uniform over time interval $(0,T_*)$; in particular, these estimates hold without assuming the condition \eqref{ass-criterion-bd2}.

We first briefly recall the a priori energy estimates and we refer to Section 3 in \cite{Barrett-Lu-Suli} for the details of the derivation for a slightly modified model.

For any time $t\in (0,T_*)$, calculating
$$
\int_0^t \left(\int_\O \eqref{02a} \cdot \vu \, \dx +  \frac{1}{2}\int_\O \tr \eqref{04a}\,\dx \right)\dt'
$$
implies the energy inequality \eqref{energy1-f}. In fact, here an energy equality is obtained due to the smoothness of the solution. Then applying Gronwall's inequality gives the following inclusions
\ba\label{a-priori-3}
&\vr\in L^\infty(0,T_*; L^\g(\O)), \quad \vu \in L^2(0,T_*; W^{1,2}_0(\O;\R^2)), \quad \vr|\vu|^2 \in L^\infty(0,T_*; L^1(\O)),\\
&\left\{\begin{aligned}&\eta \in L^\infty (0,T_*;  L^2(\Omega)) \cap L^2 (0,T_*;  W^{1,2}(\Omega)), \quad \mbox{if}\ \de >0,\\
&\eta \log \eta \in L^\infty(0,T_*;L^1(\O)), \ \eta^{\frac{1}{2}} \in L^2(0,T_*; W^{1,2}(\O)), \quad \mbox{if} \ \de=0,
\end{aligned}\right. \\
& \tr (\TT) \in L^\infty(0,T_*; L^1(\O)).
\ea

We then take the inner product of \eqref{04a} with $\TT$ and integrate over $\O$. Direct calculation implies
\begin{align*}
%\ba\label{a-priori-tau-f1}
&\frac{1}{2}\frac{\rm d}{\dt} \int_\O  | \TT|^2 \,\dx + \e \int_\O |\nabla_x \TT|^2 \,\dx + \frac{1}{2\l}\int_\O | \TT|^2 \,\dx \\
&\quad= - \int_\O {\rm Div}_x (\vu\,\TT) : \TT \,\dx + \int_\O \left(\nabla_x \vu \,\TT + \TT\, \nabla_x^{\rm T} \vu \right) : \TT \,\dx + \frac{k}{2\lambda}\int_\O \eta \,\tr\left(\TT\right) \,\dx \\
&\quad \leq  3\, \| \nabla_x \vu\|_{L^2(\O;\R^{2\times 2})}\, \|\TT\|_{L^4(\O;\R^{2\times 2})}^2   + \frac{1}{4\l} \int_\O |\TT|^2 \,\dx +
\frac{k^2}{2\l} \int_\O \eta^2 \,\dx,
\qquad  \mbox{$\forall \, t \in (0,T_*)$}.
%\ea
\end{align*}

Applying the Gagliardo--Nirenberg inequality recalled in Lemma \ref{lem-GN} in the case of $d=2$, we have, over time interval $ (0,T_*)$, that
\ba\label{sobolev-R2}
\|\TT\|_{L^4(\O)}^2 &\leq C\,\|\TT\|_{L^2(\O)} \| \TT\|_{W^{1,2}(\O)} \leq C\,\|\TT\|_{L^2(\O)}  \left(\|\TT\|_{L^2(\O)}+ \| \nabla_x\TT\|_{L^{2}(\O)}\right).\nn
\ea
This implies
\ba\label{sobolev-R2-a}
3\, \| \nabla_x \vu\|_{L^2(\O)}\, \|\TT\|_{L^4(\O )}^2 \leq C\,\| \nabla_x \vu\|_{L^2(\O )}^2 \|\TT\|_{L^2(\O )}^2 + \frac{1}{8\l}\|\TT\|_{L^2(\O )}^2+ \frac{\e}{2}  \|\nabla_x \TT\|_{L^2(\O)}^2.\nn
\ea
We thus obtain for any $t\in (0,T_*)$:
\ba\label{a-priori-tau-f3}
&\int_\O  | \TT|^2(t,\cdot) \,\dx + \e \int_0^t\int_\O |\nabla_x \TT|^2 \,\dx \,\dt' + \frac{1}{4\l} \int_0^t \int_\O | \TT|^2 \,\dx\, \dt'   \\ &\leq  \int_\O  | \TT_0|^2 \,\dx +  C\, \int_0^t \| \nabla_x \vu\|_{L^2(\O)}^2 \|\TT\|_{L^2(\O)}^2 \dt' + \frac{k^2}{\l}\int_0^t \int_\O \eta^2 \,\dx\,\dt'.
\ea

If $\de>0$, we have from the \eqref{a-priori-3} that $\|\eta\|_{L^\infty(0,T_*;L^2(\O))}\leq C$.
We thus deduce from \eqref{a-priori-tau-f3} by using Gronwall's inequality  that
\ba\label{a-priori-tau-f4}
&\|\TT(t,\cdot)\|_{L^2(\O )}^2 +  \e  \int_0^t \int_\O |\nabla_x \TT|^2  \,\dx\,\dt'
%&\quad \leq C\left(\int_0^t \| \nabla_x \vu(t',\cdot)\|_{L^2(\O;\R^{2\times 2})}^2\,\dt',\|\TT_0\|_{L^2(\O;\R^{2\times 2})}  \right)
\leq C(E_0, \|\TT_0\|_{L^2(\O )}), \qquad  \mbox{$\forall \, t \in (0,T_*)$}.
\ea

\medskip

We now consider the case $\de=0$, $L>0$. By \eqref{a-priori-3}, there holds
\ba\label{est-tau-de=0-1}
\| \eta  \log \eta \|_{L^\infty(0,T_*;L^1(\O))} + \| \nabla_x \eta ^{\frac{1}{2}} \|_{L^2(0,T_*;L^{2}(\O))}  \leq C.
\ea
Then by \eqref{est-tau-de=0-1}, we have
\[
%\ba\label{est-tau-de=0-2}
\int_\O |\nabla_x \eta| \,\dx = \int_\O |2 \eta^{\frac{1}{2}}\nabla_x \eta^{\frac{1}{2}}| \,\dx \leq 2\, \| \eta^{\frac{1}{2}} \|_{L^2(\O)} \|\nabla_x \eta^{\frac{1}{2}} \|_{L^2(\O)} \leq C.
%\ea
\]
%and by \eqref{est-tau-de=0-1} we have that
%
%\[
%\ba\label{est-tau-de=0-3}
%\| \eta \|_{L^2(0,T_*;W^{1,1}(\O))}   \leq C.
%\ea
%\]
As $d=2$, the Sobolev embedding of $W^{1,1}(\O)$ into $L^2(\Omega)$ then gives that
\[
%\ba\label{est-tau-de=0-4}
\| \eta \|_{L^2(0,T_*;L^{2}(\O))}   \leq C.
%\ea
\]
Then by \eqref{a-priori-tau-f3} and Gronwall's inequality, we obtain the same estimate as \eqref{a-priori-tau-f4}.

Hence, there holds the uniform estimates
\ba\label{a-priori-4}
\|\TT\|_{L^\infty (0,T_*;  L^2(\Omega;\R^{2\times 2}))} + \|\TT\|_{L^2 (0,T_*;  W^{1,2}(\Omega;\R^{2\times 2}))}\leq C.
\ea

\subsection{Higher order estimates for $\eta,\ \TT$}\label{sec-higher-eta}

Based on interpolations and repeat applications of Lemma \ref{lem-parabolic-2}, we can show the following higher order estimates for the polymer number density.

\begin{proposition}\label{lem-eta-better}
For any $r\in (1,\infty)$, there holds
\be\label{est-eta-higher}
\|\eta\|_{L^{\infty}(0,T_*; L^r(\O))}+ \|\eta\|_{L^{2}(0,T_*; W^{1,r}(\O))}\leq C.
\ee
\end{proposition}

\begin{proof} The proof is of the same sprit as for Proposition 12.2 in \cite{Barrett-Lu-Suli}.
We have shown in Section \ref{sec-a-priori} that $\vu \in L^2(0,T_*;W^{1,2}(\O,\R^2))$
and $\eta \in L^\infty(0,T_*;L^1(\O)) \cap L^2(0,T_*;L^2(\O))$. By Sobolev embedding $W^{1,2}(\O)\hookrightarrow L^{\frac{1}{\delta}}(\O)$, for any $\delta\in (0,1)$, we have that
\be\label{est-eta-u-de1}
 \eta  \vu \in L^{2-\frac{\delta}{2}}(0,T_*; L^1(\O;\R^2)) \cap  L^1(0,T_*; L^{2-\frac{\delta}{2}}(\O;\R^2))\hookrightarrow L^{1+c(\delta)}(0,T_*; L^{2-\delta}(\O;\R^2)),\nn
\ee
for any $\delta\in (0,1)$ and some $c(\delta)>0$. By observing $\eta_0\in W^{2,2}(\O) \subset L^\infty(\O)$,  we can apply Lemma \ref{lem-parabolic-2} to deduce that
$$
%\be\label{est-eta-para1-2}
\eta \in L^{\infty}(0,T_*; L^{2-\delta}(\O)) \cap   L^{1+c(\delta)}(0,T_*; W^{1,2-\delta}(\O)),\quad \mbox{for any $\delta\in (0,1)$ and some $c(\delta)>0$}.
%\ee
$$
This implies furthermore that
$$
%\be\label{est-eta-u-de2-1}
\eta \vu \in L^{2-\delta}(0,T_*; L^2(\O;\R^2)) \cap  L^2(0,T_*; L^{2-\delta}(\O;\R^2)), \quad \mbox{for any $\delta\in (0,1)$}.
%\ee
$$
Applying Lemma \ref{lem-parabolic-2} once more gives
\be\label{est-eta-para2}
\eta \in L^{\infty}(0,T_*; L^2(\O)) \cap   L^{2-\delta}(0,T_*; W^{1,2}(\O)) \cap  L^2(0,T_*; W^{1,2-\delta}(\O)),\quad \mbox{for any $\delta\in (0,1)$}.\nn
\ee
This implies
$$
%\be\label{est-eta-u-de3-1}
\eta \vu \in   L^{2}(0,T_*; L^{2-\delta}(\O;\R^2)) \cap  L^{1+c(\delta)}(0,T_*; L^{\frac{1}{\delta}}(\O;\R^2)), \,\mbox{for any $\delta\in (0,1)$ and some $c(\delta)>0$}.
%\ee
$$
Again by Lemma \ref{lem-parabolic-2} we deduce
$$
%\be\label{est-eta-para3-1}
\eta \in L^{\infty}(0,T_*; L^{r}(\O)) \cap   L^{1+c(r)}(0,T_*; W^{1,r}(\O)),\quad \mbox{for any $r\in (1,\infty)$ and some $c(r)>0$}.
%\ee
$$
Again by the bound on $\vu$ and Sobolev embedding we have
$$
%\be\label{est-eta-u-de3-2}
\eta\vu \in L^{2}(0,T_*; L^{r}(\O;\R^2)), \quad \mbox{for any $r\in (1,\infty)$}.
%\ee
$$
 Finally, one more application of Lemma \ref{lem-parabolic-2} implies  \eqref{est-eta-higher}.
\end{proof}

Similarly, for the extra stress tensor, we have the following improved estimates:
\begin{proposition}\label{lem-TT-better}
For any $r\in (1,\infty)$, there holds
\be\label{est-TT-higher}
\| \TT \|_{L^{\infty}(0,T_*; L^r(\O))}+ \| \TT \|_{L^{2}(0,T_*; W^{1,r}(\O))}\leq C.
\ee
\end{proposition}

\begin{proof} We rewrite the equation in $\TT$ into its each element $T_{i,j}, \ 1\leq i,j \leq 2$:
\be\label{eq-TT-ij}
\d_t \TT_{i,j}  - \e \Delta_x \TT_{i,j} = - \Div_x (\vu\,\TT_{i,j}) + \left(\nabla_x \vu \,\TT + \TT\, \nabla_x^{\rm T} \vu \right)_{i,j} + \frac{k}{2\lambda}\eta \delta_{i,j} - \frac{1}{2\lambda} \TT_{i,j}.
\ee
By \eqref{a-priori-3} and \eqref{a-priori-4}, Sobolev embedding theorem gives
\be\label{est-TT-higher1}
\vu \in L^2(0,T_*; L^{\frac{1}{\delta}}(\O)), \quad \TT \in L^{2+c(\delta)}(0,T_*; L^{\frac{1}{\delta}}(\O)),\quad \mbox{for any $\delta\in (0,1)$ and some $c(\delta)>0$}.
\ee
Thus,
\be\label{est-TT-higher2}
\vu\,\TT_{i,j} \in L^{1+c(\delta)}(0,T_*; L^{\frac{1}{\delta}}(\O)),\quad \mbox{for any $\delta\in (0,1)$ and some $c(\delta)>0$}.
\ee
Similarly, we have
\be\label{est-TT-higher3}
\left(\nabla_x \vu \,\TT + \TT\, \nabla_x^{\rm T} \vu \right)_{i,j} \in L^{1+c(\delta)}(0,T_*; L^{2-\delta}(\O)),\quad \mbox{for any $\delta\in (0,1)$ and some $c(\delta)>0$}.\nn
\ee
By applying the Bogovski{\u \i } operator (Lemma \ref{lem-bog}), there exists $H_{i,j} \in L^{1+c(\delta)}(0,T_*; W_0^{1,2-\delta}(\O;\R^2))$ such that
\ba\label{est-Hij}
\Div_x H_{i,j} =  \left(\nabla_x \vu \,\TT + \TT\, \nabla_x^{\rm T} \vu \right)_{i,j} + \frac{k}{2\lambda}\eta \delta_{i,j} - \frac{1}{2\lambda} \TT_{i,j}.\nn
\ea
By Sobolev embedding theorem, there holds
\ba\label{est-Hij1}
 \|H_{i,j}\|_{L^{1+c(\delta)}(0,T_*; L^{\frac{1}{\delta}}(\O))} \leq C, \ \mbox{for any $\delta\in (0,1)$ and some $c(\delta)>0$}.
\ea

Thus, we can write \eqref{eq-TT-ij} as
\be\label{eq-TT-ij1}
\d_t \TT_{i,j}  - \e \Delta_x \TT_{i,j} = \Div_x (-\vu\,\TT_{i,j} + H_{i,j}).\nn
\ee
This allows us to apply Lemma \ref{lem-parabolic-2}, together with the estimates in \eqref{est-TT-higher2} and \eqref{est-Hij1} and $\TT_0 \in W^{2,2}(\O;\R^{2\times 2}) \subset L^\infty(\O;\R^{2\times 2})$, to obtain
\be\label{est-TT-higher-f1}
\| \TT_{i,j} \|_{L^{\infty}(0,T_*; L^\frac{1}{\delta}(\O))}+ \| \TT_{i,j} \|_{L^{1+ c(\delta)}(0,T_*; W^{1,\frac{1}{\delta}}(\O))}\leq C, \ \mbox{for any $\delta\in (0,1)$ and some $c(\delta)>0$}.
\ee

Again by \eqref{est-TT-higher1} and \eqref{est-TT-higher-f1}, we deduce
\ba\label{est-TT-higher4}
\vu\,\TT_{i,j} \in L^{2}(0,T_*; L^{\frac{1}{\delta}}(\O)),\quad \left(\nabla_x \vu \,\TT + \TT\, \nabla_x^{\rm T} \vu \right)_{i,j} \in L^{2}(0,T_*; L^{2-\delta}(\O)),\quad \mbox{for any $\delta\in (0,1)$}.\nn
\ea
Using the Bogovski{\u \i } operator and Lemma \ref{lem-parabolic-2}, we deduce our desired estimate \eqref{est-TT-higher}.
\end{proof}

\begin{remark}\label{rem-higher-eta}

From the proof, we see that Propositions \ref{lem-eta-better} and \ref{lem-TT-better} hold true for the case $\de=0, \ L>0$.
We remark that the estimates shown in Sections \ref{sec-a-priori} and \ref{sec-higher-eta} depend only on the initial data; more precisely, they only depend on the norms given in \eqref{ini-data-f} and the norm $\|(\eta_0, \TT_0)\|_{L^\infty(\O)}$; particularly the estimates are independent of $\|\vr\|_{L^\infty((0,T_*)\times \O)}$.

\end{remark}

\subsection{Estimates for $\vr |u|^\alpha$ with $\a>2$}
The goal of this is section is to prove, on assuming \eqref{ass-criterion-bd2}, that
\be\label{est-vru-alpha}
\|\vr |\vu|^\a\|_{L^\infty(0,T_*; L^1(\O))} \leq C<\infty, \ \mbox{for some $\a>2$}.
\ee

\medskip

As the derivation of a priori energy estimates in Section \ref{sec-a-priori}, the idea is to multiply \eqref{02a} by $\a |\vu|^{\a-2} \vu$ for some $\a>2$ and integrate in $\O$. By the fact $\d |\vu|^\a = \a |\vu|^{\a-2} \vu \cdot \d \vu$ and integral by parts, one can obtain
$$
\int_\O \d_t(\vr\vu) \cdot \a |\vu|^{\a-2} \vu \,\dx + \int_\O \Div_x(\vr\vu\otimes \vu) \cdot \a |\vu|^{\a-2} \vu \,\dx = \frac{\rm d}{\dt}\int_\O \vr |\vu|^{\a} \,\dx.
$$
Repeatedly using $\d |\vu|^\a = \a |\vu|^{\a-2} \vu \cdot \d \vu$ and integral by parts, through direct calculation, implies
\ba\label{est-vru-alpha-1}
&\frac{\rm d}{\dt}\int_\O \vr |\vu|^{\a} \,\dx + \int_\O \a |\vu|^{\a-2}\left( \mu |\nabla_x \vu|^2 + \nu |\Div_x \vu|^2 \right)\,\dx \\
&\quad + \a (\a-2) \int_\O  \left( \mu |\vu|^{\a-2} \big|\nabla_x |\vu|\big|^2 + \nu (\Div_x \vu) |\vu|^{\a-3} \vu \cdot (\nabla_x |\vu|) \right)\,\dx\\
&=\a \int_\O p(\vr) \Div_x(|\vu|^{\a-2} \vu)  \,\dx + \a \int_\O (kL\eta + \de\, \eta^2) \Div_x(|\vu|^{\a-2} \vu)  \,\dx \\
&\quad - \a \int_\O \TT : \nabla_x(|\vu|^{\a-2} \vu)  \,\dx + \a \int_\O \vr \ff \cdot |\vu|^{\a-2} \vu  \,\dx.
\ea

Observing $|\nabla_x |\vu||=||\vu|^{-1} \vu \cdot \nabla_x \vu|\leq |\nabla_x \vu|$, and taking $2<\a\leq 3$ close to $2$ such that $(\a-2)\nu \leq \frac{\mu}{2}$\footnote{ By applying Cauchy--Schwartz inequality and more precise analysis, this condition can be relaxed to $(\a-2)\nu < 4\mu$. If there holds $8\mu-\nu>0$, one may choose some $\a>3$ in this step. We remark that this condition $8\mu-\nu>0$ appears in the study of related 3D problems, for instance in \cite{SWZ11}.}, implies
$$
(\a-2) \int_\O  \nu (\Div_x \vu) |\vu|^{\a-3} \vu \cdot (\nabla_x |\vu|) \,\dx \leq  \frac{1}{2}\int_\O \mu  |\vu|^{\a-2}|\nabla_x \vu|^2\,\dx.
$$
We thus deduce from \eqref{est-vru-alpha-1} that for any $t\in (0,T_*)$:
\ba\label{est-vru-alpha-2}
&\int_\O \vr |\vu|^{\a} \,\dx + \frac{\a}{2}\int_0^t \int_\O \mu |\vu|^{\a-2}|\nabla_x \vu|^2\,\dx \,\dt' \leq \int_\O \vr_0 |\vu_0|^{\a} \,\dx \\
& + \a \int_0^t \int_\O p(\vr) \Div_x(|\vu|^{\a-2} \vu)  \,\dx\,\dt' + \a \int_0^t \int_\O (kL\eta + \de\, \eta^2) \Div_x(|\vu|^{\a-2} \vu)  \,\dx\,\dt' \\
& - \a \int_0^t \int_\O \TT : \nabla_x(|\vu|^{\a-2} \vu)  \,\dx\,\dt' + \a \int_0^t \int_\O \vr \ff \cdot |\vu|^{\a-2} \vu  \,\dx\,\dt'.
\ea

Note the fact $|\Div_x(|\vu|^{\a-2} \vu)| \leq (\a-1)|\vu|^{\a-2}|\nabla_x \vu|$. Then H\"older's inequality implies
\ba\label{est-vru-alpha-vr}
  \int_0^t \int_\O p(\vr) \Div_x(|\vu|^{\a-2} \vu)  \,\dx\,\dt' \leq C(\|\vr\|_{L^\infty((0,T_*)\times \O)}) \int_0^t \int_\O \vr |\vu|^\a  + |\nabla_x \vu|^{\frac{\a}{2}}  \,\dx \,\dt',
\ea
where the integral related to $|\nabla_x \vu|^{\frac{\a}{2}}$ is uniformly bounded in $t\in (0,T_*)$ as long as $\a\leq 4$.

We then calculate
\ba\label{est-vru-alpha-eta}
 & \int_0^t \int_\O (kL\eta + \de\, \eta^2) \Div_x(|\vu|^{\a-2} \vu)  \,\dx\,\dt' \\
 &\qquad  \leq C  \|\eta+\eta^2\|_{L^\infty(0,t;L^4(\O))} \|\nabla_x \vu\|_{L^2((0,t)\times \O)} \| |\vu|^{\a-2}\|_{L^2(0,t; L^{4}( \O))}.
\ea
 By \eqref{a-priori-3}, Proposition \ref{lem-eta-better} and Sobolev embedding, the quantity on the right-hand side of \eqref{est-vru-alpha-eta} is uniformly bounded in $t\in (0,T_*)$ as long as  $2(\a-2)\leq 2$ which is equivalent to $\a\leq 3$.

By Proposition \ref{lem-TT-better} and Sobolev embedding inequality, we have
\ba\label{est-vru-alpha-TT}
 -  \int_0^t \int_\O \TT : \nabla_x(|\vu|^{\a-2} \vu)  \,\dx\,\dt'  \leq C  \|\TT\|_{L^\infty(0,t;L^4( \O)}  \|\nabla_x \vu\|_{L^2((0,t)\times \O)} \| |\vu|^{\a-2}\|_{L^2(0,t; L^{4}( \O))},
\ea
which is uniformly bounded  in $t\in (0,T_*)$ provided $2(\a-2) \leq 2 \Leftrightarrow \a \leq 3$. Similarly,
\be\label{est-vru-alpha-ff}
\int_0^t \int_\O \vr \ff \cdot |\vu|^{\a-2} \vu  \,\dx\,\dt'\leq C \int_0^t \int_\O \vr |\vu|^{\a}   \,\dx\,\dt' + C.
\ee

Summing up estimates \eqref{est-vru-alpha-2}--\eqref{est-vru-alpha-ff}, Gronwall's inequality gives our desired estimate in \eqref{est-vru-alpha}.
%\be\label{est-vru-alpha-f}
%\|\vr |\vu|^\a\|_{L^\infty(0,T_*; L^1(\O))} + \||\vu|^{\a-2}|\nabla_x \vu|^2\|_{L^1((0,T_*)\times \O)} \leq C<\infty, \ \mbox{for some $\a>2$}.
%\ee

\subsection{Improved estimates for $\nabla_x u$}

Introduce $\vv_\vr, \ \vv_\eta, \ \vv_\tau$ solving the following Dirichlet problems of Lam\'e systems:
\ba\label{Lame-vr}
&\left\{\begin{aligned} & - \mu \Delta_x \vv_\vr - \nu \nabla_x \Div_x \vv_\vr = \nabla_x p(\vr), \quad &&\mbox{in}\ \O,\\
&\vv_\vr =0, \quad &&\mbox{on}\ \d \O,
\end{aligned}\right.\\
&\left\{\begin{aligned} & - \mu \Delta_x \vv_\eta - \nu \nabla_x \Div_x \vv_\eta = \nabla_x (k L \eta +\de \eta^2), \quad &&\mbox{in}\ \O,\\
&\vv_\eta =0, \quad &&\mbox{on}\ \d \O,
\end{aligned}\right.\\
&\left\{\begin{aligned} & - \mu \Delta_x \vv_\tau - \nu \nabla_x \Div_x \vv_\tau = -\Div_x \TT, \quad &&\mbox{in}\ \O,\\
&\vv_\tau =0, \quad &&\mbox{on}\ \d \O.
\end{aligned}\right.
\ea
By the notation introduced below Lemma \ref{lem-Lame}, we write
$$\vv_\vr= \mathcal{L}^{-1}(\nabla_x p(\vr)), \quad \vv_\eta= \mathcal{L}^{-1}(\nabla_x (kL \eta +\de \eta^2)), \quad \vv_\tau=  \mathcal{L}^{-1}(-\Div_x \TT). $$
Since $\LLL$ is a linear operator independent of the time variable $t$, we have, as long as the calculation makes sense, that
\be\label{Lame-dt}
\d_t \mathcal{L}^{-1} (\vv) = \mathcal{L}^{-1} (\d_t \vv).\nn
\ee

By \eqref{ass-criterion-bd2}, Propositions \ref{lem-eta-better} and \ref{lem-TT-better}, applying Lemma \ref{lem-Lame} gives
\ba\label{est-vv}
&\vv_\vr \in L^\infty(0,T_*;W_0^{1,r}(\O;\R^2)), \quad \vv_\eta \in L^\infty(0,T_*;W_0^{1,r}(\O;\R^2)) \cap L^2(0,T_*;W^{2,r}(\O;\R^2)),\\
&\vv_\tau \in L^\infty(0,T_*;W_0^{1,r}(\O;\R^2)) \cap L^2(0,T_*;W^{2,r}(\O;\R^2)),\qquad \mbox{for any $r\in (1,\infty)$.}
\ea

We then introduce
\be\label{def-w}
\vw := \vu - \vv, \quad \vv:=(\vv_\vr + \vv_\eta + \vv_\tau)
\ee
solving in $(0,T_*)\times \O$ the system
\be\label{eq-w}
\vr \d_t \vw - \mu \Delta_x \vw - \nu \nabla_x \Div_x \vw = - \vr \vu\cdot \nabla_x \vu - \vr \d_t \vv,
\ee
with no slip boundary condition
\be\label{bdry-w}
\vw = 0 \quad \mbox{on} \ (0,T_*)\times \d\O.
\ee

The main goal of this section is to prove the following proposition which is a inspired by Proposition 3.2 in \cite{SZ11}. The proof here is more difficult and technical due to the presence of extra terms in $\eta$ and $\TT$.
\begin{proposition}\label{prop-w}
Under the assumption \eqref{ass-criterion-bd2}, we have for some $T_1\in (0,T_*)$ that
$$
\vw\in L^\infty(0,T_*;W^{1,2}_0(\O;\R^2)) \cap L^2(0,T_*;W^{1,r}(\O;\R^2)) \cap L^2(T_1, T_*;W^{2,2}(\O;\R^2)), \quad \mbox{for any $r\in (1,\infty)$}.
$$

\end{proposition}

\begin{proof}

By \eqref{est-strong-T1} and \eqref{est-vv}, there holds for any $T_1\in (0,T_*)$ that
\be\label{est-w-00}
\vw\in L^\infty(0,T_1;W^{1,2}_0) \cap L^2(0,T_1;W^{1,r})(\O;\R^2), \quad \mbox{for any $r\in (1,\infty)$}.\nn
\ee
Thus, it is sufficient to prove for some $T_1\in (0,T_*)$, which shall be fixed later on close to $T_*$, that
\be\label{est-w-0}
\vw\in L^\infty(T_1,T_*;W^{1,2}_0) \cap L^2(T_1,T_*;W^{1,r})\cap L^2(T_1, T_*;W^{2,2})(\O;\R^2)\quad \mbox{for any $r\in (1,\infty)$}.
\ee

From \eqref{eq-w}, we deduce by direct calculation that, for any $t\in (0,T_*)$,
\ba\label{est-w-1}
& \int_\O  \vr |\d_t \vw|^2 \,\dx + \frac{1}{2}\frac{\rm d}{\dt}\int_\O \mu|\nabla_x \vw|^2 \,\dx \leq \int_\O  \vr |\d_t \vw | \, |\vu\cdot \nabla_x \vu -  \d_t \vv| \,\dx \\
 &\quad \leq \frac{1}{2}\int_\O  \vr |\d_t \vw |^2  \,\dx + \int_\O  \vr |\vu\cdot \nabla_x \vu|^2 \,\dx + \int_\O \vr |  \d_t \vv|^2  \,\dx.\nn
\ea
This gives, for any $T_1\in (0,T_*)$ and any $t\in (T_1,T_*)$, that
\ba\label{est-w-2}
&\int_\O \mu|\nabla_x \vw|^2(t,\cdot) \,\dx +  \int_{T_1}^t \int_\O  \vr |\d_t \vw|^2 \,\dx \,\dt' \\
&\leq \int_\O \mu|\nabla_x \vw|^2(T_1,\cdot) \,\dx +  2 \int_{T_1}^t \int_\O  \vr |\vu\cdot \nabla_x \vu|^2\,\dx \, \dt'+ 2 \int_{T_1}^t \int_\O \vr |  \d_t \vv|^2 \,\dx\,\dt'.
\ea

We need to estimate the last two terms in \eqref{est-w-2}. For the first one, by \eqref{est-vru-alpha},  H\"older's inequality, Young's inequality, and Gagliardo--Nirenberg inequality, we have for some $\a>2$ that
\ba\label{est-w-3}
\| \sqrt\vr |\vu\cdot \nabla_x \vu|\|_{L^2(T_1,t;L^2(\O))}&\leq C \| \vr^{\frac{1}{\a}} |\vu| \|_{L^\infty(T_1,t;L^\a(\O))} \|\nabla_x\vu\|_{L^2(T_1,t;L^{\frac{2\a}{\a-2}}(\O))}\\
&\leq C \big( 1 + \|\nabla_x\vw\|_{L^2(T_1,t;L^{\frac{2\a}{\a-2}}(\O))}\big)\\
&\leq C \big( 1 + \|\nabla_x\vw\|_{L^2(T_1,t;L^{2}(\O))}^{\th}\|\nabla_x^2\vw\|_{L^2(T_1,t;L^{2}(\O))}^{1-\th}\big)\\
&\leq C  + C_\delta \|\nabla_x\vw\|_{L^2(T_1,t;L^{2}(\O))} + \delta \|\nabla_x^2\vw\|_{L^2(T_1,t;L^{2}(\O))},
\ea
for some $\th\in(0,1)$ determined by $\a$, for any $\delta >0$ and some $C_\delta >0$.

\medskip

We now estimate the $L^2$ norm of $\d_ t \vv = \d_t \vv_\vr + \d_t \vv_\eta + \d_t \vv_\tau$. For $\d_t \vv_\vr$, we have
\ba
\d_t \vv_\vr & = \d_t \LLL (\nabla_x p(\vr)) =  \LLL (\nabla_x (p'(\vr)\d_t \vr)) = - \LLL (\nabla_x (p'(\vr)\Div_x (\vr \vu)) )\\
& = - \LLL (\nabla_x \Div_x (p(\vr) \vu) ) - \LLL (\nabla_x ((p'(\vr)\vr - p(\vr))\Div_x  \vu) ).
\nn
\ea
By \eqref{ass-criterion-bd2}, applying Lemma \ref{lem-Lame} gives
\ba\label{est-dt-vvr}
\|\d_t \vv_\vr\|_{L^2(0,t;L^2(\O))} \leq C \|\nabla_x \vu\|_{L^2(0,t;L^2(\O))} \leq C.
\ea

\medskip

It is much more complicated to estimate $\d_t \vv_\eta$. By the equation in $\eta$, we have
\ba
\d_t \vv_\eta & = \d_t \LLL (\nabla_x (kL \eta + \de \eta^2)) =  k L \LLL (\nabla_x (\d_t \eta)) + 2\de \LLL(\nabla_x (\eta \d_t \eta)) \\
& =  k L \LLL (\nabla_x (-\Div_x (\eta\vu) + \e \Delta_x \eta) ) + 2\de \LLL (\nabla_x (-\eta \Div_x (\eta\vu) + \e \eta \Delta_x \eta) ) \\
& = k L \LLL (\nabla_x (-\Div_x (\eta\vu + \e \nabla_x \eta) ) \\
&+ \de \LLL (\nabla_x (-\Div_x (\eta^2\vu))) + \de \LLL (\nabla_x (-\eta^2\Div_x \vu)) + 2\de \e \LLL (\nabla_x ( \eta \Delta_x \eta) ).
\nn
\ea
By Lemma \ref{lem-Lame}, \eqref{a-priori-3}, Proposition \ref{lem-eta-better} and Sobolev embedding, we have
\ba
&\|\LLL (\nabla_x (-\Div_x (\eta\vu + \e \nabla_x \eta) )\|_{L^2(0,t;L^2(\O))} \leq C\|\eta\vu +  \nabla_x \eta\|_{L^2(0,t;L^2(\O))} \\
& \quad \leq C  \|\eta\|_{L^\infty(0,t;L^4(\O))}  \|\vu\|_{L^2(0,t;L^4(\O))} + C \|\eta\|_{L^2(0,t;W^{1,2}(\O))}
%& \quad \leq C  \|\eta\|_{L^\infty(0,t;L^4(\O))}  \|\vu\|_{L^2(0,t;W^{1,2}(\O))} + C \|\eta\|_{L^2(0,t;W^{1,2}(\O))} \\
\leq C.\nn
\ea
Similarly,
\ba
&\|\LLL (\nabla_x (-\Div_x (\eta^2\vu)))\|_{L^2(0,t;L^2(\O))} \leq C\|\eta^2\vu \|_{L^2(0,t;L^2(\O))} \leq C,\\
&\|\LLL (\nabla_x (-\eta^2\Div_x \vu))\|_{L^2(0,t;L^2(\O))}\leq C \|\eta^2\Div_x \vu\|_{L^2(0,t;L^{\frac 32}(\O))}\leq C.
\nn
\ea
Then for $\d_t \vv_\eta$, it is left to estimate $\LLL (\nabla_x ( \eta \Delta_x \eta) )$, which is the most difficult one to estimate. Observing the fact
$$
\eta \Delta_x \eta = \frac{1}{2} (\Delta_x \eta^2 - |\nabla_x \eta|^2) =  \Div_x (\eta \nabla_x\eta) - \frac{1}{2}|\nabla_x \eta|^2,
$$
and applying Lemma \ref{lem-Lame} and Proposition \ref{lem-eta-better}, gives that
\ba\label{est-dt-veta-f1}
&\|\LLL (\nabla_x ( \eta \Delta_x \eta) )\|_{L^2(T_1,t;L^2(\O))} \\
&\quad \leq \|\LLL (\nabla_x \Div_x(\eta \nabla_x \eta))\|_{L^2(T_1,t;L^2(\O))} +  \frac{1}{2} \| \LLL (\nabla_x (|\nabla_x \eta|^2) ) \|_{L^2(T_1,t;L^2(\O))}\\
&\quad \leq C \|\eta \nabla_x \eta\|_{L^2(T_1,t;L^2(\O))} +  C\||\nabla_x \eta|^2\|_{L^2(T_1,t;L^{\frac 32}(\O))}\\
&\quad \leq C + C \|\nabla_x \eta\|_{L^4(T_1,t;L^{3}(\O))}^2.
\nn
\ea

We benefit from the parabolic equation \eqref{03a} in $\eta$, by Lemma \ref{lem-parabolic-2}, to obtain the control:
\ba\label{est-dt-veta-f2}
 \|\nabla_x \eta\|_{L^4(T_1,t;L^{3}(\O))}&\leq C \|\eta\vu\|_{L^4(T_1,t;L^{3}(\O))} \leq C \|\eta\|_{L^\infty(T_1,t;L^{12}(\O))} \|\vu\|_{L^4(T_1,t;L^{4}(\O))}\\
 &\leq C \|\vu\|_{L^\infty(T_1,t;L^{2}(\O))}^{\frac{1}{2}}\|\nabla_x \vu\|_{L^2(T_1,t;L^{2}(\O))}^{\frac{1}{2}},
\nn
\ea
where we used Gagliardo--Nirenberg inequality.  Therefore,
 \ba\label{est-dt-veta-f}
 \int_{T_1}^t \int_\O \vr |  \d_t \vv_\eta|^2 \,\dx\,\dt' \leq C \|\d_t \vv_\eta\|_{L^2(T_1,t;L^{2}(\O))}^2 \leq C
+ C \|\vu\|_{L^\infty(T_1,t;L^{2}(\O))}^2 \|\nabla_x \vu\|_{L^2(T_1,t;L^{2}(\O))}^{2}.
\ea
We remark that there is \emph{no} control for $\|\vu\|_{L^\infty(T_1,t;L^{2}(\O))}^2$ so far. We will see later on that this term can be absorbed by a positive term on the left-hand side by using the smallness of $\|\nabla_x \vu\|_{L^2(T_1,t;L^{2}(\O))}^{2}$ when $T_1$ is close to $T_*$.

\medskip

For $\d_t \vv_\tau$, direct calculation gives
\ba
\d_t \vv_\tau &= \d_t \LLL(-\Div_x \TT) =   \LLL \Div_x (-\d_t\TT)     \\
& = - \LLL \Div_x \big( -{\rm Div}_x (\vu\,\TT) + \big(\nabla_x \vu \,\TT + \TT\, \nabla_x^{\rm T} \vu \big) + \e \Delta_x \TT + \frac{k}{2\lambda}\eta  \,\II - \frac{1}{2\lambda} \TT\big).
\nn
\ea
Then by \eqref{a-priori-3}, Lemma \ref{lem-Lame} and Propositions \ref{lem-eta-better} and \ref{lem-TT-better}, we obtain
\ba\label{est-vtau-f}
& \|\d_t \vv_\tau\|_{L^2(0,t;L^2(\O))}  \leq C \big( \| |\nabla_x \vu||\TT| \|_{L^2(0,t;L^{\frac{4}{3}}(\O))} +  \| (|\vu||\TT|, \nabla_x \TT,\TT, \eta)\|_{L^2(0,t;L^2(\O))}\big)\\
& \leq C\left(\|\vu\|_{L^2(0,t;L^4(\O))}\|\TT\|_{L^\infty(0,t;L^4(\O))} +  \|\nabla_x \vu \|_{L^2(0,t;L^{2}(\O))} \|\TT\|_{L^\infty(0,t;L^4(\O))} +  1 \right) \leq C.
\ea

Using the estimates in \eqref{est-w-3}, \eqref{est-dt-vvr}, \eqref{est-dt-veta-f} and \eqref{est-vtau-f} into \eqref{est-w-2} implies
\ba\label{est-w-f-1}
&\int_\O \mu|\nabla_x \vw|^2(t) \,\dx +  \int_{T_1}^t \int_\O  \vr |\d_t \vw|^2 \,\dx \,\dt' \leq \int_\O \mu|\nabla_x \vw|^2(T_1) \,\dx + \delta \|\nabla_x^2\vw\|_{L^2(T_1,t;L^{2}(\O))} \\
&\quad + C_\delta \|\nabla_x \vw\|_{L^2(T_1,t;L^{2}(\O))}  + C \|\vu\|_{L^\infty(T_1,t;L^{2}(\O))}^2 \|\nabla_x \vu\|_{L^2(T_1,t;L^{2}(\O))}^{2} + C.
\ea

By rewriting \eqref{eq-w} into a Lam\'e system
\be\label{eq-w-Lame}
- \mu \Delta_x \vw - \nu \nabla_x \Div_x \vw = -\vr \d_t \vw - \vr \vu\cdot \nabla_x \vu - \vr \d_t \vv,
\ee
armed with no-slip boundary condition \eqref{bdry-w}, we can apply Lemma \ref{lem-Lame} to obtain
\ba\label{est-w-f-2}
\|\nabla_x^2 \vw\|_{L^2(T_1,t;L^2(\O))} \leq C \|\vr \d_t \vw + \vr \vu\cdot \nabla_x \vu + \vr \d_t \vv \|_{L^2(T_1,t;L^2(\O))}.
\ea
Again by the estimates in \eqref{est-w-3}, \eqref{est-dt-vvr}, \eqref{est-dt-veta-f} and \eqref{est-vtau-f}, we deduce from  \eqref{est-w-f-2} that
\ba\label{est-w-f-3}
&\|\nabla_x^2 \vw\|_{L^2(T_1,t;L^2(\O))} \leq C \|\sqrt\vr \d_t \vw\| \|_{L^2(T_1,t;L^2(\O))} +  C_\delta \|\nabla_x \vw\|_{L^2(T_1,t;L^{2}(\O))} \\
&\quad + C \delta \|\nabla_x^2\vw\|_{L^2(T_1,t;L^{2}(\O))}
 + C \|\vu\|_{L^\infty(T_1,t;L^{2}(\O))}^2 \|\nabla_x \vu\|_{L^2(T_1,t;L^{2}(\O))}^{2} + C.
\ea
 By choosing $\delta>0$ small such that $C \delta \leq 1/2$, we deduce from \eqref{est-w-f-3} that
\ba\label{est-w-f-4}
\|\nabla_x^2 \vw\|_{L^2(T_1,t;L^2(\O))}& \leq C \|\sqrt\vr \d_t \vw\| \|_{L^2(T_1,t;L^2(\O))} +  C \|\nabla_x \vw\|_{L^2(T_1,t;L^{2}(\O))}  \\
&\quad + C \|\vu\|_{L^\infty(T_1,t;L^{2}(\O))}^2 \|\nabla_x \vu\|_{L^2(T_1,t;L^{2}(\O))}^{2} + C.
\ea
Then plugging estimate \eqref{est-w-f-4} into \eqref{est-w-f-1} and choosing $\delta>0$ sufficient small (fixed) implies
\ba\label{est-w-f-5}
&\int_\O \mu|\nabla_x \vw|^2(t,\cdot) \,\dx +  \frac{1}{2}\int_{T_1}^t \int_\O  \vr |\d_t \vw|^2 \,\dx \,\dt' \leq \int_\O \mu|\nabla_x \vw|^2(T_1,\cdot) \,\dx \\
&\quad + C \int_{T_1}^t \int_\O |\nabla_x \vw|^2 \,\dx\,\dt'  + C \|\vu\|_{L^\infty(T_1,t;L^{2}(\O))}^2 \|\nabla_x \vu\|_{L^2(T_1,t;L^{2}(\O))}^{2} + C.
\ea

By Poincar\'e inequality and \eqref{est-vv}, we have
\ba\label{est-vu-L-infty}
\|\vu\|_{L^\infty(T_1,t;L^{2}(\O))} \leq C \|\nabla_x \vu\|_{L^\infty(T_1,t;L^{2}(\O))}
%&\leq C \left( \|\nabla_x \vv\|_{L^\infty(T_1,t;L^{2}(\O))} + \|\nabla_x \vw\|_{L^\infty(T_1,t;L^{2}(\O))} \right)\\
\leq C \left( 1 + \|\nabla_x \vw\|_{L^\infty(T_1,t;L^{2}(\O))} \right).
\ea
Together with \eqref{est-w-f-5}, we obtain for any $t\in (T_1, T_*)$ that
\ba\label{est-w-f-6}
&\int_\O \mu|\nabla_x \vw|^2(t,\cdot) \,\dx +  \frac{1}{2}\int_{T_1}^t \int_\O  \vr |\d_t \vw|^2 \,\dx \,\dt' \leq \int_\O \mu|\nabla_x \vw|^2(T_1,\cdot) \,\dx \\
&\quad + C \int_{T_1}^t \int_\O |\nabla_x \vw|^2 \,\dx\,\dt'  + C \|\nabla\vw\|_{L^\infty(T_1,t;L^{2}(\O))}^2 \|\nabla_x \vu\|_{L^2(T_1,t;L^{2}(\O))}^{2} + C.
\ea
This implies, the following nonnegative quantity
$$
\xi(t):= \|\nabla_x\vw\|_{L^\infty(T_1,t;L^{2}(\O))}^2
$$
satisfies for any $t\in (T_1,T_*)$:
\ba\label{est-w-f-7}
\xi(t) \leq \xi(T_1) + C \int_{T_1}^t \xi(t') \dt' +C \xi(t) \|\nabla_x \vu\|_{L^2(T_1,T_*;L^{2}(\O))}^{2} + C.
\ea
Since $\|\nabla_x \vu\|_{L^2(0,T_*;L^{2}(\O))} <\infty$, there holds
\be\label{small-u-T1}
 \|\nabla_x \vu\|_{L^2(T_1,T_*;L^{2}(\O))} \to 0,\quad \mbox{as $T_1\to T_*$}.\nn
\ee
Thus, by choosing $T_1\in (0,T_*)$ be close to $T_*$ such that $C \|\nabla_x \vu\|_{L^2(T_1,T_*;L^{2}(\O))}^{2} \leq 1/2,$  we deduce from \eqref{est-w-f-7} that
\ba\label{est-w-f-8}
\xi(t) \leq 2 \xi(T_1) + C \int_{T_1}^t \xi(t') \dt' + C.\nn
\ea
Gronwall's inequality implies for any $t\in (T_1,T_*)$:
\ba\label{est-w-f-9}
\xi(t) \leq  e^{C(t-T_1)}\big(2 \xi(T_1) + C \big) \Longrightarrow \|\nabla_x\vw\|_{L^\infty(T_1,t;L^{2}(\O))} \leq C.
\ea
Combining the estimates in \eqref{est-w-f-6} and \eqref{est-w-f-9} gives for any $t\in (T_1,T_*)$ that
\ba\label{est-w-f0}
\|\nabla_x\vw\|_{L^\infty(T_1,t;L^{2}(\O))}  +  \|\sqrt\vr \d_t \vw\| \|_{L^2(T_1,t;L^2(\O))} \leq C.
\ea
Together with \eqref{est-w-f-3} and \eqref{est-vu-L-infty}, we obtain
\ba\label{est-w-f00}
\|\nabla_x^2\vw\|_{L^2(T_1,t;L^{2}(\O))}\leq C,\quad \mbox{for any $t\in [T_1, T_*)$}.
\ea

By \eqref{est-w-f0} and \eqref{est-w-f00}, using Sobolev embedding theorem, we thus obtain our desired estimates in \eqref{est-w-0}. The proof is then completed.

\end{proof}

\medskip

By the estimates in \eqref{est-strong-T1} and \eqref{est-vv}, a direct corollary of Proposition \ref{prop-w} is the following:
\begin{proposition}\label{pro-vu-better} Under the assumption \eqref{ass-criterion-bd2}, we have
\be\label{est-vu-better}
\vu \in L^\infty(0,T_*;W^{1,2}_0) \cap L^2(0,T_*;W^{1,r}) \quad \mbox{for any $r\in (1,\infty)$}.\nn
\ee

\end{proposition}

\subsection{End of the proof}

We are now ready to prove the $L^\infty$ bound of $\eta$ and $\TT$. We rewrite \eqref{03a} as
\be\label{eq-eta}
\d_t \eta - \e \Delta_x \eta = - \nabla_x \eta \cdot \vu  - \eta \Div_x \vu.\nn
\ee
By Proposition \ref{lem-eta-better} and Proposition \ref{pro-vu-better}, we have for any $r\in (1,\infty)$ that
\ba\label{est-eta-new-1}
&\|- \nabla_x \eta \cdot \vu  - \eta \Div_x \vu\|_{L^2(0,T_*;L^r(\O))} \\
&\leq \|\nabla_x \eta\|_{L^2(0,T_*;L^{2r}(\O))} \| \vu\|_{L^\infty (0,T_*;L^{2r}(\O))} + \|\Div_x \vu\|_{L^2(0,T_*;L^{2r}(\O))} \| \eta\|_{L^\infty (0,T_*;L^{2r}(\O))}\leq C.\nn
\ea
This allows us to apply Lemma \ref{lem-parabolic-1} with $p=2, \ q=r$ to deduce
\ba\label{est-eta-new-2}
\|\eta\|_{L^\infty(0,T_*;W^{1,r}(\O))} + \|\partial_t\eta \|_{L^2(0,T_*; L^r(\O))} + \|\eta\|_{L^2(0,T_*; W^{2,r}(\O))} \leq C.
\ea
This implies, by choosing $r>2$ and Sobolov embedding theorem, that
\ba\label{est-eta-new-f}
\|\eta\|_{L^\infty(0,T_*;L^\infty(\O))}\leq C<\infty.\nn
\ea

\medskip

While for $\TT$, a similar argument implies
\ba\label{est-TT-new-f}
\|\TT\|_{L^\infty(0,T_*;L^\infty(\O))}\leq C<\infty.\nn
\ea

We thus obtained our desired estimate \eqref{criterion-bd} and finished the proof of Theorem \ref{thm-criterion}.

%%%%%%%%%%%%%%%%%%%%%%%%%%%%%%%%%%%%%%%%%%%%%%%%%%%%%%%%%%%%%%%%%%%%%%%%%%%%%%%%%%%%%%%%%%%%%%%%%%%%%%%%%

\section{Relative entropy}

To prove the weak-strong uniqueness stated in Theorem \ref{thm-ws}, in the same spirit as the study in \cite{Germain10} and in \cite{FJN12, FNS14} for compressible Navier--Stokes equations, we introduce a proper relative entropy and build a relative entropy inequality, for which a consequence is the weak-strong uniqueness through tedious analysis.

Firstly, we will introduce a suitable relative entropy for our compressible Oldroyd--B model.  Based on the relative entropy used in \cite{Germain10, FJN12} for compressible Navier--Stokes equations, some proper modification related to the additional terms in $\eta, \ \TT$ need to be done. This modification is not a direct result by analyzing the a priori energy estimate \eqref{energy1-f}. %, which is pretty much the case for compressible Navier--Stokes equations.
For example, the term $\tr (\TT)$ on the left-hand side of the energy estimate \eqref{energy1-f} has a sign due to the positive definite property of $\TT$, but $\tr(\TT - \tilde  \TT)$ has no sign given two positive definite matrices $\TT$ and  $\tilde \TT$. We will see later on that we do not include $\TT - \tilde \TT$ in our definition of relative entropy.

For the notation convenience, we denote
\ba\label{def-H-q-G}
H(s) := \frac{a}{\g-1} s^\g, \ \ q(s) := kL s + \de\, s^2, \ G(s) :=  \left(k L  s \log s  + \de \,s^2\right), \ \forall s\in [0,\infty)
\ea
satisfying
\ba\label{pt-H-q-G}
H'(s)s - H(s) =  p(s),\quad  G'(s) s - G(s) = q(s), \quad H''(s) = p'(s)/s,\quad G''(s) = q'(s)/s,
\ea
where $p(s) = a s^\g$ (see \eqref{pressure}) denoting the power law assumption on the pressure.

Now we introduce the relative entropy.  Let $(\vr,\vu,\eta,\TT)$ be a finite energy weak solution in the sense of Definition \ref{def-weaksl-f} and obtained in Theorem \ref{thm-weak}. Let $\tilde\vr,\ \tilde\vu,\ \tilde\eta,\ \tilde\TT$ be the so-called relative functions which have sufficient regularity. Define the following two relative entropies:
\ba\label{def-entropy}
&\mathcal{E}_1(t)=\mathcal{E}_1(\vr,\vu,\tilde\vr,\tilde \vu)(t):=\int_\O \frac{1}{2} \vr |\vu-\tilde \vu|^2 + \left(H(\vr) - H(\tilde \vr) - H'(\tilde \vr) (\vr-\tilde \vr)\right) (t,\cdot)\,\dx,\\
&\mathcal{E}_2(t)=\mathcal{E}_2(\eta,\teta)(t) :=\int_\O  \left(G(\eta) - G(\tilde \eta) - G'(\tilde \eta) (\eta-\tilde \eta)\right)(t,\cdot) \,\dx.
\ea

\begin{remark}

The relative entropy $\CalE_1$ is the same as in \cite{Germain10, FJN12}. The new one $\CalE_2$ is built in a similar manner. We remark that the extra stress tensor $\TT$ is not included in the relative entropies. One reason is explained above that $\tr\,(\TT - \tilde \TT)$ has no sign. Another reason is that we do not want to make the remainder term $\mathcal{R}$ in the relative entropy inequality shown later on in Proposition \ref{prop-entropy} be too massy. While this is enough to show the weak-strong uniqueness. Indeed, as we shall see later on in Section \ref{sec-thm-ws-unique}, based on the relative entropy inequality obtained in this section, together with a $L^2$ type estimate for $\TT-\tilde \TT$, we can derive the weak-strong uniqueness.

\end{remark}

We give some properties that we will use for the quantities appearing in the relative entropies:
\begin{lemma}\label{lem-H-G}
 There exists $\delta>0, \ c>0$ depending only on $a$ and $\g$ such that for any $\vr, \tvr \geq 0,$
\be\label{pt-H}
H(\vr) - H(\tilde \vr) - H'(\tilde \vr) (\vr-\tilde \vr) \geq  \left\{\begin{aligned} & c \tvr^{\g-2} (\vr-\tvr)^2, \ && \mbox{if $\delta \tvr \leq \vr \leq \delta^{-1} \tvr$},\\
&c \max\{\vr^\g, \tvr^\g\}, \ &&\mbox{otherwise.}
\end{aligned}\right.
\ee

For any $\eta, \teta \geq 0$, there holds
\be\label{pt-G}
G(\eta) - G(\tilde \eta) - G'(\tilde \eta) (\eta-\tilde \eta) \geq 2 \de (\eta-\teta)^2 + \left\{\begin{aligned} &  \frac{k L (\eta-\teta)^2}{2\tilde \eta}, \ && \mbox{if $\eta \leq 2 \tilde \eta$},\\
& \frac{k L \eta}{4}, \ &&\mbox{if $\eta \geq 2 \tilde \eta$.}
\end{aligned}\right.
\ee
\end{lemma}

\begin{proof}[Proof of Lemma \ref{lem-H-G}]
We recall that $a>0, \ \g>1$. We use the definition of $H$ in \eqref{def-H-q-G} to obtain
\ba\label{pt-H2}
H(\vr) - H(\tilde \vr) - H'(\tilde \vr) (\vr-\tilde \vr) = \frac{a}{\g-1} \vr^\g - \frac{a\g}{\g-1} \tvr^{\g-1} \vr + a\tvr^\g.
\ea
Let $0 < \delta \leq 1/2$. We suppose $\vr \leq \delta \tvr$. Then
\ba\label{pt-H3}
&H(\vr) - H(\tilde \vr) - H'(\tilde \vr) (\vr-\tilde \vr)  =  a\tvr^\g + \frac{a}{\g-1} \left( \vr^\g - \g  \tvr^{\g-1} \vr\right) \\
& \geq a\tvr^\g + \frac{a}{\g-1} \big( (\delta\tvr)^\g - \g  \tvr^{\g-1} (\delta\tvr)\big) = a \tvr^\g \big(1 + \frac{1}{\g-1}(\delta^\g - \delta \g)\big),
\ea
where we used the fact that the function $f(\vr) := \vr^\g - \g  \tvr^{\g-1} \vr $ is decreasing for $\vr \in [0,\tvr]$. % by observing $f'(\vr) = \g \vr(\vr^{\g-1} - \tvr^{\g-1} ) \leq 0$ for $\vr \in [0,\tvr]$.
The limit
$
\lim_{\delta \to 0} (\delta^\g - \delta \g) =0
$
implies that there exists some $\delta\in (0,\frac 12)$ depending only on $\g$ such that
$$
1 + \frac{1}{\g-1}(\delta^\g - \delta \g)\geq \frac 12.
$$
Together with \eqref{pt-H3}, we obtain
\ba\label{pt-H4}
H(\vr) - H(\tilde \vr) - H'(\tilde \vr) (\vr-\tilde \vr) \geq  a \tvr^\g/2, \ \mbox{if $\vr\leq \delta \tvr$}.
\ea

We then consider the case  $\vr \geq \delta^{-1} \tvr$. By \eqref{pt-H2}, we have
\ba\label{pt-H5}
&H(\vr) - H(\tilde \vr) - H'(\tilde \vr) (\vr-\tilde \vr) = \frac{a}{\g-1} \vr^\g  + a \big(\tvr^\g - \frac{\g}{\g-1} \tvr^{\g-1} \vr \big) \\
& \geq \frac{a}{\g-1} \vr^\g +  a \big((\delta\vr)^\g - \frac{\g}{\g-1} (\delta\vr)^{\g-1} \vr \big) = \frac{a}{\g-1} \vr^\g \big( 1 +  (\g-1)\delta^\g - \g \delta^{\g-1} \big),\nn
\ea
where we used the fact that the function $g(\tvr) :=\tvr^\g - \frac{\g}{\g-1} \tvr^{\g-1} \vr $ is decreasing for $\tvr \in [0,\vr]$. The limit
$
\lim_{\delta\to 0} (\g-1)\delta^\g - \g \delta^{\g-1}  =0
$
implies for some $\delta>0$ small and determined by $\g$ that
$$
 1 +  (\g-1)\delta^\g - \g \delta^{\g-1} \geq \frac 12.
$$
Thus, for such fixed $\delta,$
\ba\label{pt-H6}
H(\vr) - H(\tilde \vr) - H'(\tilde \vr) (\vr-\tilde \vr) \geq  \frac{a}{2(\g-1)} \vr^\g, \ \mbox{if $\vr\geq \delta^{-1} \tvr$}.
\ea

Now we consider the case $ \delta \tvr \leq  \vr \leq \delta^{-1} \tvr$. By Taylor's formula, direct calculation gives
$$
H(\vr) - H(\tilde \vr) - H'(\tilde \vr) (\vr-\tilde \vr) = H''(\hat \vr) (\vr - \tvr)^2 = a \g \hat \vr ^{\g-2} (\vr - \tvr)^2
$$
for some $\hat \vr$ between $\vr$ and $\tvr$. Thus,
\ba\label{pt-H7}
H(\vr) - H(\tilde \vr) - H'(\tilde \vr) (\vr-\tilde \vr) \geq a \g \min\{\delta^{\g-2},\delta^{2-\g}\} \tvr^{\g-2} (\vr - \tvr)^2 , \ \mbox{if $\delta \tvr \leq  \vr \leq \delta^{-1} \tvr$}.
\ea

Summing up the estimates in \eqref{pt-H4}, \eqref{pt-H6} and \eqref{pt-H7} gives our desired result \eqref{pt-H}.

\medskip

By Taylor's formula, we have
$$
G(\eta) - G(\tilde \eta) - G'(\tilde \eta) (\eta-\tilde \eta) = 2\de (\eta-\teta)^2 + k L \,\hat\eta^{-1} (\eta-\teta)^2,
$$
for some $\hat \eta$ between $\eta$ and $\teta$. This implies directly our desired estimate \eqref{pt-G}.
\end{proof}

\medskip

We now state the relative entropy inequality in our setting.
\begin{proposition}\label{prop-entropy}
Let $T>0$ and $\O \subset \R^2$ be a $C^{2,\b}$ domain with $\b\in (0,1)$. Let $(\vr,\vu,\eta,\TT)$ be a finite energy weak solution in the sense of Definition \ref{def-weaksl-f} and obtained in Theorem \ref{thm-weak}. Let $\tilde\vr,\ \tilde\vu,\ \tilde\eta$ be smooth functions in $(t,x)\in [0,T]\times \overline\O$ with constrains
\be\label{cond-r-eta}
\tvu=0 \ \mbox{on} \ [0,T]\times \d\O, \quad \inf_{(0,T)\times \O }{\tvr} >0,\quad \inf_{(0,T)\times \O }{\teta} >0.\nn
\ee
Then there holds the following \emph{relative entropy inequality}: for any $t\in (0,T]$,
\ba\label{ineq-entropy}
& \CalE_1(\vr,\vu,\tilde\vr,\tilde \vu)(t)  + \CalE_2(\eta, \teta)(t) + \int_0^t  \int_\O \mu \left| \nabla_x (\vu-\tvu)  \right|^2 + \nu |\Div_x (\vu-\tvu) |^2\,\dx \,\dt' \\
& \qquad \qquad \qquad \qquad  \qquad  + 2 \e \int_0^t \int_\O 2 k L   |\nabla_x (\eta^{\frac12}-\teta^{\frac12}) |^2 +  \de\,  |\nabla_x (\eta-\teta)|^2 \,\dx\,\dt'\\
&\quad \leq \CalE_1(\vr_0,\vu_0,\tilde\vr_0,\tilde \vu_0)  + \CalE_2(\eta_0, \teta_0) + \int_0^t \RR (t') \,\dt',
\ea
where $(\vr_0,\vu_0,\tilde\vr_0,\tilde \vu_0,\eta_0,\teta_0)$ denotes the corresponding initial values and $\RR(t) = \sum_{j=1}^5\RR_j(t)$ with
\ba\label{Rj-def}
\RR_1(t) & : = \int_\O \vr(\d_t \tvu + \vu \cdot \nabla_x \tvu)\cdot (\tvu-\vu)\,\dx\\
& + \int_\O \mu \nabla_x \tvu :\nabla_x (\tvu-\vu) + \nu \Div_x \tvu \,\Div_x (\tvu-\vu)\,\dx + \int_\O \vr \ff\cdot(\vu-\tvu)\,\dx\\
& + \int_\O (\tvr-\vr) \d_t H'(\tvr) + (\tvr\tvu -\vr\vu) \cdot \nabla_x H'(\tvr)\,\dx + \int_\O \Div_x \tvu (p(\tvr)-p(\vr))\,\dx,\\
\RR_2(t)&: =   \int_\O (\teta-\eta) \d_t G'(\teta) + (\teta\tvu -\eta\vu) \cdot \nabla_x G'(\teta)\,\dx + \int_\O \Div_x \tvu (q(\teta)-q(\eta))\,\dx,\\
\RR_3(t)&: = - 4 \e k L \int_\O  \nabla_x \teta^{\frac 12} \cdot \nabla_x (\eta^{\frac 12}-\teta^{\frac 12}) + \nabla_x \eta^{\frac 12} \cdot \nabla_x \teta^{\frac 12}(1-\teta^{-\frac 12}\eta^{\frac 12})\,\dx,\\
\RR_4(t)&: = - 2 \e \de \int_\O  \nabla_x \teta \cdot \nabla_x (\eta-\teta) \,\dx,\\
\RR_5(t)&: =   \int_\O \TT : \nabla_x (\tvu-\vu)\,\dx.
\ea

\end{proposition}

\begin{proof}[Proof of Proposition \ref{prop-entropy}]
We calculate
\be\label{vr-vu-tvu1}
\int_\O \frac{1}{2} \vr |\vu-\tilde \vu|^2 \,\dx = \int_\O \frac{1}{2} \vr |\vu|^2 \,\dx - \int_\O \vr \vu\cdot \tilde \vu \,\dx + \int_\O \frac{1}{2} \vr |\tvu|^2 \,\dx,
\ee
where for the second and third terms we have, by taking $\tvu$ as a test function in the weak formulation of the momentum equation \eqref{weak-form3-f} and taking $\frac{1}{2} |\tvu|^2$ as a test function in the weak formulation of the continuity equation \eqref{weak-form1-f}, that
\ba\label{vr-vu-tvu2}
 &\int_\O \vr \vu\cdot \tilde \vu \,\dx  =  \intO{ \vr_{0} \vu_{0} \cdot \tvu_0} + \int_0^t \intO{ \big[ \vr \vu \cdot \partial_t \tvu + (\vr \vu \otimes \vu) : \Grad \tvu  + p(\vr)\, \Div_x \tvu \\
&\quad + q(\eta)\, \Div_x \tvu  - (\mu \nabla_x \vu : \Grad \tvu + \nu \Div_x \vu \,\Div_x \tvu )- \TT : \nabla_x\tvu + \vr\, \ff \cdot \tvu \big] }\, \dt'
\ea
and
\ba\label{vr-vu-tvu3}
 \int_\O \frac{1}{2}\vr| \tvu|^2 \,\dx %&= \int_\O \frac{1}{2}\vr_0 | \tvu_0|^2 \,\dx  + \int_0^t \int_\O \frac 12 \vr \d_t |\tvu|^2 + \frac 12 \vr \vu \cdot \nabla_x |\tvu|^2\,\dx\,\dt'  \\&
 = \int_\O \frac{1}{2}\vr_0 | \tvu_0|^2 \,\dx  + \int_0^t \int_\O  \vr \tvu \cdot \d_t \tvu + \vr \vu \cdot \nabla_x \tvu \cdot \tvu \,\dx\,\dt'.
\ea

Similarly, testing the continuity equation by $H'(\tvr)$ gives
\ba\label{vr-vu-tvu4}
 \int_\O \vr H'(\tvr) \,\dx & = \int_\O \vr_0 H'(\tvr_0) \,\dx  + \int_0^t \int_\O  \vr \d_t H'(\tvr) + \vr \vu \cdot \nabla_x H'(\tvr) \,\dx\,\dt'.
\ea

Plugging \eqref{vr-vu-tvu2} and \eqref{vr-vu-tvu3} into \eqref{vr-vu-tvu1} and using \eqref{vr-vu-tvu4} and the energy inequality \eqref{energy1-f} gives
\ba\label{vr-vu-tvu5}
&\int_\O \frac{1}{2} \vr |\vu-\tilde \vu|^2 + \big(H(\vr) - H'(\tvr) \vr\big)\, \dx + \int_0^t  \int_\O \mu \left| \nabla_x (\vu-\tvu)  \right|^2 + \nu |\Div_x (\vu-\tvu) |^2\,\dx \,\dt' \\
& \quad +\int_\O  G(\eta) + \frac{1}{2}\tr \left(\TT \right)\dx + 2\e\int_0^t \int_\O 2k L   |\nabla_x \eta^{\frac12}|^2 +  \de\,  |\nabla_x \eta|^2 \,\dx\,\dt'  + \frac{1}{4\l} \int_0^t  \int_\O \tr\left(\TT\right) \dx \,\dt' \\
&\leq \int_\O \frac{1}{2} \vr_0 |\vu_0-\tilde \vu_0|^2 + \big( H(\vr_0) - H'(\tvr_0) \vr_0 \big)\dx + \int_\O  G(\eta_0)+ \frac{1}{2}\tr \left(\TT_0\right) \dx + \frac{k }{2\lambda} \int_0^t \int_\O \eta  \,\dx \,\dt'\\
&\quad + \int_0^t \int_\O \vr\,\ff \cdot  (\vu-\tvu) \,\dx\,\dt'  + \int_0^t\int_\O \mu \nabla_x \tvu : \nabla_x (\tvu-\vu) +\nu \Div_x \tvu \,\Div(\tvu-\vu)\,\dx\,\dt' \\
&\quad - \int_0^t \intO{ \big[ \vr \vu \cdot (\partial_t \tvu +  \vu\cdot \Grad \tvu)  + p(\vr)\, \Div_x \tvu + q(\eta)\, \Div_x \tvu  - \TT : \nabla_x\tvu \big] }\, \dt'\\
&\quad + \int_0^t \int_\O  \vr \tvu \cdot \d_t \tvu + \vr \vu \cdot \nabla_x \tvu \cdot \tvu \,\dx\,\dt' - \int_0^t \int_\O  \vr \d_t H'(\tvr) + \vr \vu \cdot \nabla_x H'(\tvr) \,\dx\,\dt'.
\ea

By \eqref{pt-H-q-G}, we have
\ba\label{pt-H8}
&\vr \d_t H'(\tvr) + \vr \vu \cdot \nabla_x H'(\tvr)\\
 &= \tvr \d_t H'(\tvr) + \tvr \tvu \cdot \nabla_x H'(\tvr) + (\vr-\tvr) \d_t H'(\tvr) + (\vr \vu -\tvr\tvu) \cdot \nabla_x H'(\tvr)\\
&= (\vr-\tvr) \d_t H'(\tvr) + (\vr \vu -\tvr\tvu) \cdot \nabla_x H'(\tvr) + \d_t p(\tvr) + \tvu\cdot \nabla_x p(\tvr) \\
 &= (\vr-\tvr) \d_t H'(\tvr) + (\vr \vu -\tvr\tvu) \cdot \nabla_x H'(\tvr) + \d_t p(\tvr) + \Div_x (p(\tvr)\tvu) - p(\tvr) \Div \tvu.
\ea
Taking test function $\YY= \II/2$ in the weak formulation \eqref{weak-form4-f} implies
\ba\label{a-priori-tau}
 \int_\O \frac{1}{2}\tr (\TT) \,\dx  + \frac{1}{4\l}\int_0^t \int_\O \tr\left(\TT
\right) \,\dx\,\dt'
=  \int_\O \frac{1}{2}\tr \left(\TT_0\right) \,\dx + \int_0^t \int_\O \frac{k}{2\lambda} \,\eta  +  \TT : \nabla_x \vu \,\dx\,\dt'.
\ea

Using \eqref{pt-H8} and \eqref{a-priori-tau} in \eqref{vr-vu-tvu5}, together with the fact $p(\tvr) = H'(\tvr)\tvr - H(\tvr)$, implies
\ba\label{vr-vu-tvu6}
&\CalE_1(t) +\int_\O  G(\eta) \dx + \int_0^t  \int_\O \mu \left| \nabla_x (\vu-\tvu)  \right|^2 + \nu |\Div_x (\vu-\tvu) |^2\,\dx \,\dt'   \\
&\qquad + 2\e\int_0^t \int_\O 2k L   |\nabla_x \eta^{\frac12}|^2 +  \de  |\nabla_x \eta|^2 \,\dx\,\dt'  \\
&\leq \CalE_1(0) + \int_\O  G(\eta_0)\dx + \int_0^t \int_\O \vr\,\ff \cdot  (\vu-\tvu) \,\dx\,\dt'  \\
&\quad + \int_0^t\int_\O \mu \nabla_x \tvu : \nabla_x (\tvu-\vu) +\nu \Div_x \tvu \,\Div(\tvu-\vu)\,\dx\,\dt'  - \int_0^t \intO{  q(\eta)\, \Div_x \tvu   }\, \dt'\\
&\quad + \int_0^t \intO{ \big[  p(\tvr)- p(\vr)\big] \Div_x \tvu}\,\dt'- \int_0^t \int_\O (\vr-\tvr) \d_t H'(\tvr) + (\vr \vu -\tvr\tvu) \cdot \nabla_x H'(\tvr) \,\dx\,\dt' \\
&\quad + \int_0^t \int_\O  \vr  (\d_t \tvu +  \vu \cdot \nabla_x \tvu)\cdot (\tvu-\vu)\,\dx\,\dt' + \int_0^t \intO{ \TT : \nabla_x(\vu-\tvu) }\, \dt'.
\ea

\medskip

Now we include $\CalE_2= \CalE_2(\eta,\teta)$. We take $G'(\teta)$ as a test function in \eqref{weak-form2-f} to obtain
\be\label{eta-eta1}
\intO{ \eta G'(\teta) }  = \intO{ \eta_{0} G'(\teta_0) } + \int_0^t \intO{ \big[ \eta \partial_t G'(\teta)  + \eta \vu \cdot \Grad G'(\teta)  - \e \nabla_x\eta \cdot \nabla_x G'(\teta)  \big]} \, \dt' .
\ee
We deduce from \eqref{pt-H-q-G} that
\ba\label{eta-eta2}
&\eta \d_t G'(\teta) + \eta \vu \cdot \nabla_x G'(\teta)\\
& = \teta \d_t G'(\teta) + \teta \tvu \cdot \nabla_x G'(\teta) + (\eta-\teta) \d_t G'(\teta) + (\eta \vu -\teta\tvu) \cdot \nabla_x G'(\teta)\\
&= (\eta-\teta) \d_t G'(\teta) + (\eta \vu -\teta\tvu) \cdot \nabla_x G'(\teta) + \d_t q(\teta) + \tvu\cdot \nabla_x q(\teta) \\
 &= (\eta-\teta) \d_t G'(\teta) + (\eta \vu -\teta\tvu) \cdot \nabla_x G'(\teta) + \d_t q(\teta) + \Div_x (q(\teta)\tvu) - q(\teta) \Div \tvu.
\ea
By \eqref{vr-vu-tvu6}, \eqref{eta-eta1} and \eqref{eta-eta2}, and the fact $q(\teta)=G'(\teta) \teta - G(\teta)$, we have
\ba\label{vr-vu-tvu-eta1}
&\CalE_1(t) + \CalE_2(t) + \int_0^t  \int_\O \mu \left| \nabla_x (\vu-\tvu)  \right|^2 + \nu |\Div_x (\vu-\tvu) |^2\,\dx \,\dt' \\
&\quad + 2\e\int_0^t \int_\O 2k L   |\nabla_x \eta^{\frac12}|^2 +  \de\,  |\nabla_x \eta|^2 \,\dx\,\dt' \\
&\leq \CalE_1(0) +\CalE_2(0)  + \int_0^t \int_\O \vr\,\ff \cdot  (\vu-\tvu) \,\dx\,\dt'  \\
&\quad + \int_0^t\int_\O \mu \nabla_x \tvu : \nabla_x (\tvu-\vu) +\nu \Div_x \tvu \,\Div(\tvu-\vu)\,\dx\,\dt' \\
&\quad + \int_0^t \intO{ \big[  p(\tvr)- p(\vr)\big] \Div_x \tvu}\dt'- \int_0^t \int_\O (\vr-\tvr) \d_t H'(\tvr) + (\vr \vu -\tvr\tvu) \cdot \nabla_x H'(\tvr) \,\dx\,\dt' \\
&\quad + \int_0^t \intO{  \big[q(\teta)-q(\eta)\big]\, \Div_x \tvu   }\, \dt' - \int_0^t \intO{ (\eta-\teta) \d_t G'(\teta) + (\eta \vu -\teta\tvu) \cdot \nabla_x G'(\teta)  }\, \dt'\\
&\quad + \int_0^t \int_\O  \vr  (\d_t \tvu +  \vu \cdot \nabla_x \tvu)\cdot (\tvu-\vu)\,\dx\,\dt' + \int_0^t \intO{ \TT : \nabla_x(\vu-\tvu) }\, \dt'\\
&\quad + \int_0^t \intO{  \e \nabla_x\eta \cdot \nabla_x G'(\teta) } \, \dt'.
\ea

By \eqref{def-H-q-G}, there holds
\be\label{eta-eta3}
\nabla_x \eta \cdot \nabla_x G'(\teta) = \nabla_x \eta \cdot (kL \tilde\eta^{-1} + 2\de) \nabla_x \tilde\eta = k L  \tilde\eta^{-1} \nabla_x \eta \cdot \nabla_x \tilde\eta + 2\de \nabla_x \eta  \cdot \nabla_x \tilde\eta.
\ee
We then calculate:
\be\label{eta-eta4}
4|\nabla_x \eta^{\frac12}|^2 - \tilde\eta^{-1} \nabla_x \eta \cdot \nabla_x \tilde\eta = 4|\nabla_x (\eta^{\frac12}-\teta^{\frac 12})|^2 + 4 \nabla_x \teta^{\frac{1}{2}} \cdot \nabla_x(\eta^{\frac12}-\teta^{\frac 12}) + 4 \nabla_x \eta^{\frac12}\cdot \nabla_x \teta^{\frac 12}(1-\eta^{\frac12}\teta^{-\frac 12}),
\ee
and
\be\label{eta-eta5}
|\nabla_x \eta|^2 -  \nabla_x \eta  \cdot \nabla_x \tilde\eta = |\nabla_x (\eta-\tilde\eta)|^2 + \nabla_x \tilde\eta \cdot \nabla_x (\eta-\tilde\eta).
\ee
Finally, plugging \eqref{eta-eta3}--\eqref{eta-eta5} into \eqref{vr-vu-tvu-eta1} gives our desired inequality \eqref{ineq-entropy}.% The proof is completed.

\end{proof}

\begin{remark} By the proof of Proposition \ref{prop-entropy}, we see that the regularity constrains on $(\tvr,\tvu,\teta)$ can be relaxed accordingly, as long as all the integrals in \eqref{ineq-entropy} make sense.

\end{remark}

\section{Weak-strong uniqueness}\label{sec-thm-ws-unique}

This section is devoted to proving Theorem \ref{thm-ws}. We shall employ the relative entropy inequality shown in the last section to achieve such a goal. Let $(\vr, \vu, \eta, \TT)$ be the finite energy weak solution obtained in Theorem \ref{thm-weak} and $(\tvr, \tvu, \teta, \tTT)$ be the strong solution obtained in Theorem \ref{thm-strong} with the same initial data satisfying the lower bound constrain \eqref{bound-lower}. Then for any $T<T_*$, by the continuity equation \eqref{01a}, we have
\ba\label{lower-vr-eta}
   \inf_{[0, T]\times \O} \tvr  \geq   e^{-\int_0^{T} \|\Div_x \tvu (t)\|_{L^\infty(\O)}\dt}  \inf_{\O} \vr_0 >0,\quad \inf_{[0, T]\times \O} \teta \geq  e^{-\int_0^{T} \|\Div_x \tvu(t)\|_{L^\infty(\O)}\dt}  \inf_{\O} \eta_0 >0.
\ea

%Suppose the lower bound assumption in \eqref{bound-lower} is satisfied for some $T\in (0,T_*)$. In the rest of this section, we restrict the time variable $t\in [0,T]$.

Let $T<T_*$ be arbitrary and fixed. In the rest of this section we restrict $t\in [0,T]$. Thus, by \eqref{lower-vr-eta}, we can choose $(\tvr,\tvu,\teta)$ as the relative functions in the entropy inequality \eqref{ineq-entropy}. We then analyze the corresponding right-hand side of \eqref{ineq-entropy} until some level that allows us to use Gronwall type inequalities to show the relative entropy is identically zero, which implies the weak solution and the strong one are equal. This is done in the rest of this section step by step.

\subsection{A new expression for the remainder}

Since $(\tvr,\tvu,\teta,\tTT)$ is the strong solution to \eqref{01a}--\eqref{07a} satisfying \eqref{lower-vr-eta}, we have
\be\label{eq-tvu}
\d_t \tvu + \tvu \cdot \nabla_x \tvu  + \tvr^{-1}\nabla_x p(\tvr) - \tvr^{-1}( \mu \Delta_x \tvu + \nu\nabla_x \Div_x \tvu) = \tvr^{-1} \Div_x \tTT - \tvr^{-1} \nabla_x q(\teta)  + \ff.
\ee
Plugging \eqref{eq-tvu} into $\RR_1$ in \eqref{Rj-def}, together with the fact $\tvr^{-1}\nabla_x p(\tvr)=\nabla_x H'(\tvr)$, gives
\ba\label{R1-1}
\RR_1 & =  \int_\O \vr (\vu-\tvu) \cdot \nabla_x \tvu \cdot (\tvu-\vu)\,\dx + \int_\O (\mu \Delta_x \tvu + \nu \nabla_x \Div_x \tvu) (\tvr^{-1}\vr-1) \cdot (\tvu-\vu) \,\dx \\
& - \int_\O \vr \nabla_x H'(\tvr) \cdot (\tvu-\vu)\,\dx  - \int_\O \vr \tvr^{-1} \nabla_x q(\teta) \cdot (\tvu-\vu)\,\dx + \int_\O \vr \tvr^{-1} \Div_x \tTT \cdot (\tvu-\vu)\,\dx\\
& + \int_\O (\tvr-\vr) \d_t H'(\tvr) + (\tvr\tvu -\vr\vu) \cdot \nabla_x H'(\tvr)\,\dx + \int_\O \Div_x \tvu (p(\tvr)-p(\vr))\,\dx.
\nn
\ea
By the continuity equation \eqref{01a}, we have
\ba\label{R1-2}
 & - \vr \nabla_x H'(\tvr) \cdot (\tvu-\vu) +  (\tvr-\vr) \d_t H'(\tvr) + (\tvr\tvu -\vr\vu) \cdot \nabla_x H'(\tvr)\\
 &= (\tvr-\vr) \left( \d_t H'(\tvr) + \tvu \cdot \nabla_x H'(\tvr)\right) \\
 &= (\tvr-\vr) \left[ \d_t H'(\tvr) + \Div_x(\tvu H'(\tvr)) + (H''(\tvr)\tvr - H'(\tvr)) \Div_x \tvu\right] -  (\tvr-\vr) H''(\tvr)\tvr \Div_x \tvu \\
 &=-(\tvr-\vr) p'(\tvr) \Div_x \tvu.
\ea

Then we can write $\RR_1: =\sum_{j=1}^6\RR_{1,j}$ with
\ba\label{R1-3}
 &\RR_{1,1}:= \int_\O \vr (\vu-\tvu) \cdot \nabla_x \tvu \cdot (\tvu-\vu)\,\dx,\\
 & \RR_{1,2}:= \int_\O (\mu \Delta_x \tvu + \nu \nabla_x \Div_x \tvu) \tvr^{-1} (\vr-\tvr) \cdot (\tvu-\vu) \,\dx,\\
 &\RR_{1,3}:=  \int_\O \Div_x \tvu \big(p(\tvr)-p(\vr)- p'(\tvr)(\tvr-\vr) \big)\,\dx, \\
 &\RR_{1,4}:= \int_\O\tvr^{-1} (\tvr-\vr ) \nabla_x q(\teta) \cdot (\tvu-\vu)\,\dx + \int_\O \tvr^{-1} (\vr  - \tvr) \Div_x \tTT \cdot (\tvu-\vu)\,\dx,\\
  &\RR_{1,5}:= -\int_\O \nabla_x q(\teta) \cdot (\tvu-\vu)\,\dx = -\int_\O \teta \nabla_x G'(\teta) \cdot (\tvu-\vu)\,\dx ,\\
   &\RR_{1,6}:=  \int_\O  \Div_x \tTT \cdot (\tvu-\vu)\,\dx = -\int_\O   \tTT : \nabla_x (\tvu-\vu)\,\dx.
\ea
By \eqref{Rj-def}, we have
\ba\label{R156-R2-R5}
&\RR_{1,5}+ \RR_2= \int_\O (\teta-\eta) \left(\d_t G'(\teta) + \tvu  \cdot \nabla_x G'(\teta)\right)\,\dx + \int_\O \Div_x \tvu (q(\teta) - q(\eta))\,\dx,\\
&\RR_{1,6}+ \RR_5 = \int_\O  (\TT - \tTT) : \nabla_x (\tvu-\vu)\,\dx.
\ea

By equation \eqref{03a} and similar calculations as in \eqref{R1-2},
\ba\label{R15-R2}
\RR_{1,5}+ \RR_2 & =  \int_\O \Div_x \tvu \big(q(\teta)-q(\eta) - q'(\teta) (\teta -\eta)\big)\,\dx + \e \int_\O (\tilde \eta -\eta)\tilde \eta^{-1} q'(\tilde\eta) \Delta_x\tilde \eta\,\dx\\
& = \RR_{2,1} + \e \int_\O (\tilde \eta -\eta) (k L \tilde\eta^{-1} + 2\de) \Delta_x\tilde \eta\,\dx,
\ea
where
\be\label{def-R21}
\RR_{2,1}:= \int_\O \Div_x \tvu \big(q(\teta)-q(\eta) - q'(\teta) (\teta -\eta)\big)\,\dx.
\ee
 Thus, direct calculation gives
\ba\label{R15-R2-R4}
&\RR_{1,5}+ \RR_2 + \RR_3 + \RR_4  = \RR_{2,1} + \e k L \int_\O (\tilde \eta -\eta) \tilde\eta^{-1} \Delta_x\tilde \eta\,\dx\\
 & \quad - 4 \e k L \int_\O  \nabla_x \teta^{\frac 12} \cdot \nabla_x (\eta^{\frac 12}-\teta^{\frac 12}) + \nabla_x \eta^{\frac 12} \cdot \nabla_x \teta^{\frac 12}(1-\teta^{-\frac 12}\eta^{\frac 12})\,\dx\\
 & = \RR_{2,1} +   \e k L \int_\O 4 \teta^{-\frac 12} ( \teta^{\frac 12} - \eta^{\frac 12}) \nabla_x \teta^{\frac 12} \cdot \nabla_x (\teta^{\frac 12}-\eta^{\frac 12}) - \teta^{-1} \Delta_x \teta (\eta^{\frac 12}-\teta^{\frac 12})^2\,\dx.
\ea
Summarizing the calculations in \eqref{R1-3}--\eqref{R15-R2-R4}, we obtain a new expression for the remainder $\RR$:
\ba\label{RR-new}
\RR &= \int_\O \vr (\vu-\tvu) \cdot \nabla_x \tvu \cdot (\tvu-\vu)\,\dx + \int_\O (\mu \Delta_x \tvu + \nu \nabla_x \Div_x \tvu) \tvr^{-1} (\vr-\tvr) \cdot (\tvu-\vu) \,\dx \\
&\quad + \int_\O \Div_x \tvu \big(p(\tvr)-p(\vr)- p'(\tvr)(\tvr-\vr) \big)\,\dx + \int_\O \Div_x \tvu \big(q(\teta)-q(\eta) - q'(\teta) (\teta -\eta)\big)\,\dx \\
&\quad + \int_\O\tvr^{-1} (\tvr - \vr) \nabla_x q(\teta) \cdot (\tvu-\vu)\,\dx + \int_\O \tvr^{-1} (\vr  - \tvr) \Div_x \tTT \cdot (\tvu-\vu)\,\dx\\
&\quad  +  \e k L \int_\O 4 \teta^{-\frac 12} ( \teta^{\frac 12} - \eta^{\frac 12}) \nabla_x \teta^{\frac 12} \cdot \nabla_x (\teta^{\frac 12}-\eta^{\frac 12}) - \teta^{-1} \Delta_x \teta (\eta^{\frac 12}-\teta^{\frac 12})^2\,\dx \\
&\quad + \int_\O  (\TT - \tTT) : \nabla_x (\tvu-\vu)\,\dx.
\ea

\subsection{Estimate for the remainder}\label{sec-est-RR}

We estimate the right-hand side of \eqref{RR-new} term by term. For notation convenience, let $\zeta(t)$ be some nonnegative function that is integrable over $[0,T]$; its value may differ from line to line.

By \eqref{est-strong} and Sobolev embedding $W^{2,6}(\O)\subset W^{1,\infty}(\O)$, we have
\ba\label{RR-new-1}
\int_\O \vr (\vu-\tvu) \cdot \nabla_x \tvu \cdot (\tvu-\vu)\,\dx  \leq \|\nabla_x \tvu(t)\|_{L^\infty(\O)} \int_\O \vr |\vu-\tvu|^2\,\dx \leq  \zeta(t) \CalE_1(t).
\ea

\medskip

By \eqref{est-strong} and Sobolev embedding $W^{2,6}(\O)\subset W^{1,\infty}(\O)$, we have $\Div_x \tvu\in L^2(0,T;L^\infty(\O))$. Together with the fact
\ba
&p(\tvr)-p(\vr)- p'(\tvr)(\tvr-\vr) = (\g-1)^{-1} \big[ H(\tvr)-H(\vr)- H'(\tvr)(\tvr-\vr)\big], \\
&q(\teta)-q(\eta) - q'(\teta) (\teta -\eta) = 2\de (\teta-\eta)^2 \leq G(\eta)-G(\teta) - G'(\teta) (\eta-\teta),
\nn
\ea
we thus have
\ba\label{RR-new-2}
&\int_\O \Div_x \tvu \big(p(\tvr)-p(\vr)- p'(\tvr)(\tvr-\vr) \big)\,\dx + \int_\O \Div_x \tvu \big(q(\teta)-q(\eta) - q'(\teta) (\teta -\eta)\big)\,\dx \\
&\quad \leq \zeta(t) \big(\CalE_1(t) + \CalE_2(t)\big).
\ea

\medskip

By \eqref{est-strong} (or by the argument of proving \eqref{est-eta-new-2} using Lemma \ref{lem-parabolic-1}), we have, by Sobolev embedding, for any $r\in (2,\infty)$,
\ba\label{est-teta-1}
\teta\in  L^\infty (0,T; W^{1,r}) \cap L^2(0,T; W^{2,r})(\O) \subset L^\infty(0,T; L^{\infty})\cap L^2(0,T; W^{1,\infty}(\O)).\nn
\ea
Together with the lower bound for $\teta$ in \eqref{lower-vr-eta}, we have
\ba\label{RR-new-3}
&\e k L \int_\O 4 \teta^{-\frac 12} ( \teta^{\frac 12} - \eta^{\frac 12}) \nabla_x \teta^{\frac 12} \cdot \nabla_x (\teta^{\frac 12}-\eta^{\frac 12})\,\dx \\
&\leq 2 \e k L \|\teta^{-1}(t)\|_{L^\infty(\O)} \|\nabla_x \teta(t)\|_{L^\infty(\O)} \|(\teta^{\frac 12} - \eta^{\frac 12})(t)\|_{L^2(\O)} \|\nabla_x (\teta^{\frac 12} - \eta^{\frac 12})(t)\|_{L^2(\O)}\\
&\leq 4 \e k L \|\teta^{-1}(t)\|^2_{L^\infty(\O)} \|\nabla_x \teta(t)\|^2_{L^\infty(\O)} \|(\teta^{\frac 12} - \eta^{\frac 12})(t)\|^2_{L^2(\O)} + \e k L\|\nabla_x (\teta^{\frac 12} - \eta^{\frac 12})(t)\|_{L^2(\O)}^2\\
& \leq  \zeta(t)  \|(\teta^{\frac 12} - \eta^{\frac 12})(t)\|^2_{L^2(\O)}  + \e k L\|\nabla_x (\teta^{\frac 12} - \eta^{\frac 12})(t)\|_{L^2(\O)}^2.
\ea
By \eqref{pt-G}, we have the bound
\be\label{pt-G-2}
(\teta^{\frac 12} - \eta^{\frac 12})^2 \leq 4 \left( G(\eta) - G(\tilde \eta) - G'(\tilde \eta) (\eta-\tilde \eta)\right),
\ee
which is actually uniform in $\eta, \teta \in (0,\infty)$. Together with \eqref{RR-new-3}, we have
\ba\label{RR-new-4}
\e k L \int_\O 4 \teta^{-\frac 12} ( \teta^{\frac 12} - \eta^{\frac 12}) \nabla_x \teta^{\frac 12} \cdot \nabla_x (\teta^{\frac 12}-\eta^{\frac 12})\,\dx  \leq   \zeta(t) \CalE_2(t)  + \e k L\|\nabla_x (\teta^{\frac 12} - \eta^{\frac 12})(t)\|_{L^2(\O)}^2.
\ea
By using \eqref{pt-G-2}, we have the estimate
\ba\label{RR-new-5}
&\e k L \int_\O   - \teta^{-1} \Delta_x \teta (\eta^{\frac 12}-\teta^{\frac 12})^2\,\dx \\
&\leq \e k L \int_\O    \teta^{-2} | \nabla_x \teta|^2 (\eta^{\frac 12}-\teta^{\frac 12})^2\,\dx - \e k L \int_\O  2 \teta^{-1} \nabla_x \teta \cdot  \nabla (\eta^{\frac 12}-\teta^{\frac 12}) (\eta^{\frac 12}-\teta^{\frac 12})\,\dx\\
& \leq  \|\teta^{-1} \nabla_x \teta (t)\|^2_{L^\infty(\O)} \int_\O  \e k L  (\eta^{\frac 12}-\teta^{\frac 12})^2\,\dx\\
& \quad + 4 \e k L  \|\teta^{-1}(t)\nabla_x \teta(t)\|^2_{L^\infty(\O)} \|(\teta^{\frac 12} - \eta^{\frac 12})(t)\|^2_{L^2(\O)} + \e k L\|\nabla_x (\teta^{\frac 12} - \eta^{\frac 12})(t)\|_{L^2(\O)}^2\\
& \leq    \zeta(t) \CalE_2(t)  + \e k L\|\nabla_x (\teta^{\frac 12} - \eta^{\frac 12})(t)\|_{L^2(\O)}^2.
\ea

We now consider
\ba\label{RR-new-7}
 \int_\O \tvr^{-1} (\tvr - \vr) \nabla_x q(\teta) \cdot (\tvu-\vu)\,\dx = I_1 + I_2 + I_3,
\ea
where, for $\delta$ be chosen as in Lemma \ref{lem-H-G},
\ba\label{RR-new-7-1}
&I_1 := \int_{\delta \tvr\leq \vr\leq \delta^{-1} \tvr}  \tvr^{-1} (\tvr - \vr) \nabla_x q(\teta) \cdot (\tvu-\vu) \,\dx,\\
 &I_2 := \int_{\vr\geq \delta^{-1} \tvr}  \tvr^{-1} (\tvr - \vr) \nabla_x q(\teta) \cdot (\tvu-\vu) \,\dx,\\
 &I_3 := \int_{\vr\leq \delta \tvr}  \tvr^{-1} (\tvr - \vr) \nabla_x q(\teta) \cdot (\tvu-\vu) \,\dx.
\ea

By \eqref{pt-H}, \eqref{est-strong}, the lower bound of $\tvr$ in \eqref{lower-vr-eta} and Sobolev embedding, we have for some $\sigma>0$ small that
\ba\label{RR-new-7-2}
I_1 & \leq  C(\sigma) \| \nabla_x q(\teta) \|_{L^3(\O)}^2  \|(\vr-\tvr)\|_{L^2(\delta \tvr\leq \vr\leq \delta^{-1} \tvr)}^2   + \sigma \| (\tvu-\vu)\|_{L^6(\O)}^2\\
& \leq \zeta (t) \int_{\delta \tvr\leq \vr\leq \delta^{-1} \tvr} H(\vr)- H(\tvr) - H'(\tvr)(\vr-\tvr)\,\dx + \frac{\mu}{16} \| \nabla(\tvu-\vu)\|_{L^2(\O)}^2.
\ea

Similarly, for $I_2$,
\ba\label{RR-new-7-3}
I_2 & \leq  C \int_{\vr\geq \delta^{-1} \tvr} |\nabla_x q(\teta)| \vr |\tvu-\vu|\,\dx \leq  C \int_{\vr\geq \delta^{-1} \tvr}\vr |\nabla_x q(\teta)|^2  \,\dx + C \int_{\O}  \vr |\tvu-\vu|^2 \,\dx \\
& \leq  C |\nabla_x q(\teta)|^2_{L^\infty(\O)} \int_{\vr\geq \delta^{-1} \tvr}\vr^\g   \,\dx + C \int_{\O}  \vr |\tvu-\vu|^2 \,\dx\\
& \leq \zeta(t) \int_{\vr\geq \delta^{-1} \tvr}  H(\vr)- H(\tvr) - H'(\tvr)(\vr-\tvr)   \,\dx + C \int_{\O}  \vr |\tvu-\vu|^2 \,\dx.
\ea

And for $I_3$,
\ba\label{RR-new-7-4}
I_3 & \leq  C \int_{\vr\leq \delta \tvr} | \nabla_x q(\teta)| \, | (\tvu-\vu)| \,\dx \leq C(\sigma) \| \nabla_x q(\teta) \|_{L^3(\O)}^2  \|1\|_{L^2(\vr\leq \delta \tvr)}^2   + \sigma \| (\tvu-\vu)\|_{L^6(\O)}^2\\
& \leq \zeta (t) \int_{\vr\leq \delta \tvr} H(\vr)- H(\tvr) - H'(\tvr)(\vr-\tvr)\,\dx + \frac{\mu}{16} \| \nabla(\tvu-\vu)\|_{L^2(\O)}^2.
\ea

Summing up \eqref{RR-new-7-2}--\eqref{RR-new-7-4}, we obtain
\ba\label{RR-new-7-f}
 \int_\O \tvr^{-1} (\tvr - \vr) \nabla_x q(\teta) \cdot (\tvu-\vu)\,\dx \leq \zeta (t) \CalE_1(t) + \frac{\mu}{8} \| \nabla(\tvu-\vu)\|_{L^2(\O)}^2.
\ea
The similar argument implies
\ba\label{RR-new-8-f}
\int_\O \tvr^{-1} (\vr  - \tvr) \Div_x \tTT \cdot (\tvu-\vu)\,\dx \leq \zeta (t) \CalE_1(t) + \frac{\mu}{8} \| \nabla(\tvu-\vu)\|_{L^2(\O)}^2
\ea
and
\ba\label{RR-new-6-f}
\int_\O (\mu \Delta_x \tvu + \nu \nabla_x \Div_x \tvu) \tvr^{-1} (\vr-\tvr) \cdot (\tvu-\vu) \,\dx \leq \zeta (t) \CalE_1(t) + \frac{\mu}{8} \| \nabla(\tvu-\vu)\|_{L^2(\O)}^2.
\ea
We remark that, unlike $\nabla_x q(\teta)$ or $\Div_x \tTT$, we do not have a control for the $L^2(0,T;L^\infty(\O))$ norm of $\Delta_x \tvu$ in Theorem \ref{thm-strong}. Thus a little more steps need to be taken concerning the estimate \eqref{RR-new-6-f}; precisely, we need to modify the estimate of the following term compared to \eqref{RR-new-7-3}:
\ba\label{RR-new-6-1}
\int_{\vr\geq \delta^{-1} \tvr} (\mu \Delta_x \tvu + \nu \nabla_x \Div_x \tvu) \tvr^{-1} (\vr-\tvr) \cdot (\tvu-\vu) \,\dx.\nn
\ea
Actually, we have
\ba\label{RR-new-6-2}
&\int_{\vr\geq \delta^{-1} \tvr} (\mu \Delta_x \tvu + \nu \nabla_x \Div_x \tvu) \tvr^{-1} (\vr-\tvr) \cdot (\tvu-\vu) \,\dx\\
&\leq C \int_{\vr\geq \delta^{-1} \tvr} \vr |\mu \Delta_x \tvu + \nu \nabla_x \Div_x \tvu|^2\,\dx + C\int_\O \vr |\tvu-\vu|^2 \,\dx\\
& \leq C \|\vr\|_{L^\g(\vr\geq \delta^{-1} \tvr)} \|\nabla_x^2\tvu\|_{L^{\frac{2\g}{\g-1}}}^2+ C\int_\O \vr |\tvu-\vu|^2 \,\dx.
\ea
By the low bound on $\tvr$ in \eqref{lower-vr-eta}, we have
$$
\|\vr\|_{L^\g(\vr\geq \delta^{-1} \tvr)} \leq C \|\vr\|_{L^\g(\vr\geq \delta^{-1} \tvr)}^\g \leq C \int_{\vr\geq \delta^{-1} \tvr}  H(\vr)- H(\tvr) - H'(\tvr)(\vr-\tvr)   \,\dx.
$$
Then, by the estimate on $\nabla_x^2\tvu$ in \eqref{est-strong}, we deduce from \eqref{RR-new-6-2} that:
\ba\label{RR-new-6-3}
&\int_{\vr\geq \delta^{-1} \tvr} (\mu \Delta_x \tvu + \nu \nabla_x \Div_x \tvu) \tvr^{-1} (\vr-\tvr) \cdot (\tvu-\vu) \,\dx\\
&\leq \zeta(t) \int_{\vr\geq \delta^{-1} \tvr}  H(\vr)- H(\tvr) - H'(\tvr)(\vr-\tvr)   \,\dx + C\int_\O \vr |\tvu-\vu|^2 \,\dx \leq \zeta (t) \CalE_1(t).\nn
\ea

\medskip

H\"older's inequality implies
\ba\label{RR-new-9-f}
 \int_\O  (\TT - \tTT) : \nabla_x (\tvu-\vu)\,\dx \leq  C \int_\O  |\TT - \tTT|^2\,\dx + \frac{\mu}{8} \| \nabla_x (\tvu-\vu)\|_{L^2(\O)}^2
 \ea

\medskip

Finally, summarizing the estimates in \eqref{RR-new-1}, \eqref{RR-new-2}, \eqref{RR-new-4}, \eqref{RR-new-5}, \eqref{RR-new-7-f},  \eqref{RR-new-8-f},  \eqref{RR-new-6-f} and \eqref{RR-new-9-f}, we deduce from \eqref{RR-new} that
 \ba\label{RR-new-f}
&\RR (t) \leq \zeta(t) (\CalE_1+\CalE_2)(t) + \frac{\mu}{2}  \| \nabla(\tvu-\vu)\|_{L^2(\O)}^2 + 2 \e k L\|\nabla_x (\teta^{\frac 12} - \eta^{\frac 12})\|_{L^2(\O)}^2 + C \int_\O  |\TT - \tTT|^2\,\dx.
\ea
It is left to deal with $\int_\O  |\TT - \tTT|^2\,\dx$.

\subsection{End of the proof}\label{sec-est-RR-T}
%The task here is to estimate $$\int_\O  (\TT - \tTT) : \nabla_x (\tvu-\vu)\,\dx.$$

First of all, since the initial data are assumed to be regular enough as in Theorem \ref{thm-strong}, we can employ Propositions \ref{lem-eta-better} and \ref{lem-TT-better} to obtain better estimates for $\eta$ and $\TT$:
\be\label{est-eta-T-higher}
(\eta,\TT)\in L^{\infty}(0,T_*; L^r(\O))\cap L^{2}(0,T_*; W^{1,r}(\O)),\quad \mbox{for any $r\in (1,\infty)$}.
\ee
This allows us to take $\TT$ as a test function in the weak formulation \eqref{weak-form4-f}. We can also take the strong solution $\tTT$ as a test function in \eqref{weak-form4-f}. Together with the equation in $\tTT$, through tedious but rather direct calculations, we obtain
\ba\label{est-T-tT}
& \int_\O \frac{1}{2} |\TT-\tTT|^2\, \dx + \frac{1}{2\l} \int_0^t \int_\O |\TT-\tTT|^2\, \dx \,dt' + \e \int_0^t  \int_\O |\nabla_x(\TT-\tTT)|^2\, \dx\,\dt'\\
&= - \int_0^t \int_\O \big[{\rm Div}_x\big( (\vu-\tvu) \TT \big) + {\rm Div}_x\big( \tvu (\TT-\tTT) \big)    \big] : (\TT-\tTT) \, \dx \,\dt' \\
&+ \int_0^t \int_\O  \big[\nabla_x (\vu-\tvu) \TT  + \nabla_x^{\rm T}\tvu (\TT-\tTT) \big] : (\TT-\tTT) \, \dx \,\dt' + \frac{k}{2\l} \int_0^t \int_\O(\eta-\teta)\,\tr\,(\TT-\tTT)\,\dx\,\dt' .
\nn
\ea

Let $\zeta(t)$ denote an arbitrary function in $L^1([0,T])$. We use \eqref{est-eta-T-higher} to get the estimate
\ba\label{est-T-tT-1}
& -\int_\O {\rm Div}_x\big( (\vu-\tvu) \TT \big)  : (\TT-\tTT) \, \dx  = - \int_\O \big( {\rm Div}_x (\vu-\tvu) \TT + (\vu-\tvu)\cdot \nabla_x \TT \big)  : (\TT-\tTT) \, \dx\\
&\quad \leq \big(\| \Div_x (\vu-\tvu)  \|_{L^2(\O)} \|\TT \|_{L^\infty(\O)} + \|\vu-\tvu\|_{L^6(\O)} \|\nabla_x \TT\|_{L^3(\O)}  \big)\|\TT-\tTT \|_{L^2(\O)}\\
&\quad \leq \zeta(t) \|\TT-\tTT \|_{L^2(\O)}^2 + \frac{\mu}{8} \| \nabla_x (\vu-\tvu)  \|_{L^2(\O)}^2.
\ea
Similarly as \eqref{est-T-tT-1}, we have
\ba\label{est-T-tT-2}
-\int_\O  {\rm Div}_x\big( \tvu (\TT-\tTT) \big) : (\TT-\tTT) \, \dx  = - \frac{1}{2}\int_\O (\Div_x \tvu) |\TT-\tTT|^2 \, \dx \leq \| \Div_x \tvu  \|_{L^\infty(\O)} \|\TT-\tTT \|_{L^2(\O)}^2.\nn
\ea

The other terms can be estimated similarly. At last, we arrive at
\ba\label{est-T-tT-f}
&\int_\O \frac{1}{2} |\TT-\tTT|^2\, \dx + \e \int_0^t  \int_\O |\nabla_x(\TT-\tTT)|^2\, \dx\,\dt' \\
& \quad \leq \zeta(t) \|\TT-\tTT \|_{L^2(\O)} + \frac{k}{2\l} \|\eta-\teta\|_{L^2(\O)}^2 + \frac{\mu}{4} \| \nabla_x (\vu-\tvu)  \|_{L^2(\O)}^2
\ea

\medskip

Denote
\be\label{def-E3}
 \CalE(t) := \CalE_1(t)+ \CalE_2(t) + \int_\O  \frac{1}{2}|\TT-\tTT|^2(t,\cdot)\, \dx.
\ee
Thus, by the estimates \eqref{RR-new-f} and \eqref{est-T-tT-f}, by Proposition \ref{prop-entropy}, we derive for any $t\in (0,T]$:
\ba\label{est-entropy}
&\CalE(t) + \int_0^t \int_\O \frac{\mu}{4} | \nabla_x (\vu-\tvu)  |^2 + 2 \e k L|\nabla_x (\teta^{\frac 12} - \eta^{\frac 12})|^2 +  2 \e \de |\nabla_x (\teta - \eta)|^2 + \e |\nabla_x(\TT-\tTT)|^2\, \dx\,\dt' \\
&\quad \leq     \int_0^t \zeta(t')\CalE (t') \,\dt',\quad \mbox{for some $\zeta(t)\in L^1([0,T])$}.\nn
\ea
Gronwall's inequality gives $\CalE(t) \equiv 0$ for $t\in [0,T]$ for any $T\in (0,T_*)$ which implies the weak-strong uniqueness \eqref{weak=strong}. The proof of Theorem \ref{thm-ws} is completed.

\subsection{Conditional regularity}\label{sec-thm-con-reg}

 At last, we prove Theorem \ref{thm-con-reg} concerning the conditional regularity for weak solutions. This is indeed a consequence of the refined blow-up criterion and the weak-strong uniqueness. By Theorem \ref{thm-strong}, we let $(\tilde\vr, \tilde \vu, \teta,\tTT)$ be the strong solution with $T_*$ the maximal existence time issued from the same initial data as the weak solution $(\vr,\vu,\eta,\TT)$.

\medskip

We firstly show that $T<T_*$.  By contradiction we assume $T \geq T_*$. Then for any $T_1<T_*$, since $\tvr$ and $\teta$ has a positive lower bound over $[0,T_1]\times \O$ (see \eqref{lower-vr-eta}), we can apply Theorem \ref{thm-ws} to derive that the weak solution coincides with the strong one over $[0,T_1]$. This implies
$$\sup_{[0,T_1]\times \O} \tvr = \sup_{[0,T_1]\times \O} \vr \leq    \sup_{[0,T)\times \O} \vr < \infty, $$
where the upper bound is uniform as $T_1\to T_*$. Thus,
$$\sup_{[0,T_*)\times \O} \tvr  \leq    \sup_{[0,T)\times \O} \vr < \infty.$$
This implies, by Theorem \ref{thm-criterion}, $T_* = \infty$ which contradicts $T\geq T_*$.

Now we have $T<T_*$. We can choose $T_1$ such that $T\leq T_1<T_*$. Since $\tvr$ and $\teta$ has a positive lower bound over $[0,T_1]\times \O$ (see \eqref{lower-vr-eta}), we can apply Theorem \ref{thm-ws} to derive that the weak solution is indeed the strong one over $[0,T_1] \supset [0,T]$. The proof of Theorem \ref{thm-con-reg} is then completed.

%%%%%%%%%%%%%%%%%%%%%%%%%%%%%%%%%%%%%%%%%%%%%%%%%%%%%%%%%%%%%%%%%%%%%%%%%%%%%%%%%%%%%%%%%%%%%%%%%%%%%%%%%

\end{document}